\documentclass[11pt,a4]{amsart}

\usepackage{comment}
\usepackage{amsmath,amsthm,amssymb,graphicx,fancyhdr,mathtools}
\usepackage{color}
\usepackage[top=1in, bottom=1in, left=1in, right=1in]{geometry}
\usepackage{enumerate}
\usepackage[titletoc,title]{appendix}
\usepackage{bbm}

\usepackage[all,cmtip]{xy}

\newtheorem{theorem}{Theorem}[section]
\newtheorem{definition}{Definition}[section]
\newtheorem{lemma}{Lemma}[section]
\newtheorem{corollary}{Corollary}[section]
\newtheorem{assumption}{Assumption}
\newtheorem{remark}{Remark}[section]
\newtheorem{proposition}{Proposition}[section]
   \newtheoremstyle{example}{\topsep}{\topsep}%
     {}
     {}
     {\bfseries}
     {}
     {\newline}
     {\thmname{#1}\thmnumber{ #2}\thmnote{ #3}}

   \theoremstyle{example}

\newcounter{example}[section]
\newenvironment{example}[1][]{\refstepcounter{example}\par\medskip
   \noindent \textbf{Example~\theexample. #1} \rmfamily}{\medskip}

\newcommand{\E}{\mathbb{E}}
\newcommand{\C}{\mathcal{C}}
\newcommand{\F}{\mathcal{F}}

\newcommand{\N}{\mathbb{N}}
\newcommand{\Prob}{\mathbb{P}}
\newcommand{\Poly}{\mathcal{P}}

\newcommand{\R}{\mathbb{R}}

\newcommand{\Nu}{\mathcal{V}}
\newcommand{\Tau}{\mathcal{T}}
\newcommand{\NN}{\mathtt{N}}
\newcommand{\X}{\mathbf{X}}
\newcommand{\PP}{\mathbf{P}}

\numberwithin{equation}{section}
\author{P. Patie}\thanks{The authors are grateful to M.Savov for discussion related to long-tailed distributions.}
\address{School of Operations Research and Information Engineering, Cornell University, Ithaca, NY 14853.}
\email{	pp396@orie.cornell.edu}

\author{A. Srapionyan}
\address{Center for Applied Mathematics, Cornell University, Ithaca, NY 14853.}
\email{	as3348@cornell.edu}

\usepackage{hyperref}

\title{Spectral projections correlation structure for short-to-long range dependent processes}

\begin{document}
\begin{abstract}
Let $\X=(\X_t)_{t \geq 0}$ be a stochastic process issued from $x \in \R$ that admits a marginal stationary measure $\nu$, i.e.~$\nu \PP_t f = \nu f$ for all $t \geq 0$, where $\PP_t f(x)= \E_x[f(\X_t)]$. In this paper we introduce the (resp.~biorthogonal) spectral projections correlation functions which are expressed in terms of projections into the eigenspaces of $\PP_t$ (resp.~and of its adjoint in the weighted Hilbert space $L^2(\nu)$). We obtain closed-form expressions involving eigenvalues, the condition number and/or the angle between the projections in the following different situations: when $\X=X$ with $X=(X_t)_{t \geq 0}$ being a Markov process, $\X$ is the subordination of $X$ in the sense of Bochner, and $\X$ is a non-Markovian process which is obtained by time-changing $X$ with an inverse of a subordinator. It turns out that these spectral projections correlation functions have different expressions with respect to these classes of processes which enables to identify substantial and deep properties about their dynamics. This interesting fact can be used to design original statistical tests to make inferences, for example, about the path properties of the process (presence of jumps), distance from symmetry (self-adjoint or non-self-adjoint) and short-to-long-range dependence. To reveal the usefulness of our results, we apply them to a class of non-self-adjoint Markov semigroups studied in~\cite{patie2015spectral}, and then time-change by subordinators and their inverses.
\end{abstract}


\maketitle
\section{Introduction} \label{intro}
Stochastic processes play an important role in the investigation of random phenomena depending on time. When using a stochastic process for modeling or for statistical testing purposes, one should take into account its special features which indicate how well the process reflects the reality. Some of the most essential features include (but are not limited to) observing whether the process is Markovian or not, whether its trajectories are continuous or incorporate jumps, what type of range dependence it exhibits, and how far it is from symmetry (self-adjointness).

With the objective in mind, we  introduce the concept of (biorthogonal) spectral projections correlation functions, see Definition~\ref{definition} below. We proceed by computing explicitly these functions along with their large time asymptotic behavior  for three classes of processes, namely Markov processes, Markov processes subordinated in the sense of Bochner and non-Markovian processes which are obtained by time-changing a Markov process with an inverse of a subordinator. These findings  enable us to provide a unified and original framework for designing statistical tests that  investigates critical properties of a stochastic process including the one described above. Indeed, in these three scenarios the (biorthogonal) spectral projections correlation functions have different expressions, involving some quantities characterizing the process, such as their eigenvalues with their associated condition number or the angle between the spectral projections.

We indicate that the recent years have witnessed the ubiquity of such non-Markovian dynamics in relation to the fractional Cauchy problem, see e.g.~ \cite{toaldo, orshingher2018, hairer_etal}, and, also due to their central role in diverse physical applications within the field of anomalous diffusion, see e.g.~\cite{meerschaert_sikorskii}, as well as for neuronal models for which their long range dependence feature is attractive, see e.g.~\cite{levakova2015review}. We also mention that Leonenko et~al.~\cite{leonenko2013correlation} and Mijena and Nane~\cite{mijena2014corr} investigate the orthogonal spectral projections correlation structure in the framework of Pearson diffusions, i.e.~diffusions with polynomial  coefficients. More specifically, in~\cite{leonenko2013correlation}, the authors discuss the case when a Pearson diffusion is time-changed by an inverse of an $\alpha$-stable subordinator, $0 <\alpha<1$. Whereas the authors of ~\cite{mijena2014corr} consider a Pearson diffusion time-changed by an inverse of a linear combination of independent $\alpha$- and $\beta$-stable subordinators, $0<\alpha,\beta<1$. In this work, we start with a general Markov process that admits an invariant measure with its associated semigroup not necessarily being self-adjoint and local, and then we perform a time-change with general subordinators and their inverses.

Finally, we emphasize that the notion of long-range dependence, also known as long memory, of stochastic processes has been and it is still a center of great interests in probability theory and its applications in the last decades. We refer for thorough and historical account of this concept to the recent monograph of Samorodnitsky~\cite{samorodnitsky2016}. The definitions of long-range dependence based on the second-order properties of a stationary stochastic process such as asymptotic behavior of covariances, spectral density, and variances of partial sums are among the most developed ones appearing in literature. These second-order properties are conceptually relatively simple and easy to estimate from the data. By far the most popular point of view on range dependence is through the rate of decay of covariance or correlation functions. Conceptually, short memory corresponds to a sufficiently fast rate of decay of the correlation (covariance) function as geometric decay, and long-range dependence corresponds to a sufficiently slow rate of decay of the correlation (covariance) function as power decay.

\subsection{Preliminaries}\label{sec:prelim}
Let $\X=(\X_t)_{t \geq 0}$ be a stochastic process defined on a sample filtered probability space $(\Omega, \F, (\F_t)_{t\geq0}, \Prob)$ and state space $E \subseteq \R$, endowed with a sigma-algebra $\mathcal{E}$. Let its associated family of linear operators $\PP=(\PP_t)_{t \geq 0}$ defined, for any $t \geq 0$ and $f \in \mathcal{B}_b(E)$, the space of bounded Borelian functions on $E$, by
\[ \PP_t f(x) = \E_x[f(\X_t)], \]
where $\E_x$ stands for the expectation operator with respect to $\Prob_x(\X_0 = x) =1$. Since $x \mapsto \E_x$ is $\mathcal{E}$-measurable, for any Radon measure $\nu$, we use the notation
\[\nu \PP_t f = \E_{\nu}[f(\X_t)] = \int_E \E_x[f(\X_t)]\nu(dx). \]
We say that a Radon measure $\nu$ on $E$ is a \emph{marginal stationary measure}, if for all $t \geq 0$,
\begin{equation}\label{eq:inv_mes}
\nu \PP_t f = \nu f.
\end{equation}
Note that if $\X$ is a Markov process and \eqref{eq:inv_mes} holds, we say that $\nu$ is an \emph{invariant measure}.
Then, since $\nu$ is non-negative on $E$, we define the weighted Hilbert space
\[L^2(\nu) = \lbrace f: E \rightarrow \R \text{ measurable}; \int_E f^2(x)\nu(dx) <\infty  \rbrace, \]
endowed with the inner product $\langle \cdot,\cdot \rangle_{\nu}$, where
$\langle f,g \rangle_{\nu}=\int_0^{\infty} f(x)g(x)\nu(dx)$, and norm $\Vert f \Vert_{\nu} = \sqrt{\langle f, f \rangle}_{\nu}$. Next, the operators $\PP_t$, $t \geq 0$ being linear, positive and with total mass $\PP_t \mathbbm{1} = \mathbbm{1}$ with $\mathbbm{1}$ being the identity function on the appropriate space, we have, by Jensen's inequality, for any $f \in C_0(E) \subseteq \mathcal{B}_b(E)$ where $C_0(E)$ is the set of continuous functions on $E$ vanishing at infinity,
\begin{equation*}
 \Vert \PP_t f \Vert_{\nu}^2 = \int_E(\PP_t f)^2(x)\nu(dx) \le \int_E \PP_t f^2(x) \nu(dx) = \nu f^2.
\end{equation*}
Thus, the Hahn-Banach theorem yields that we can extend $\PP_t$ as a contraction of $L^2(\nu)$.
From now on, when there is no confusion, we denote by $\PP_t$ its extension to $L^2(\nu)$.
Now, let $\PP^*=(\PP^*_t)_{t \geq 0}$ be the adjoint of $\PP$ in $L^2(\nu)$, i.e.~for any $t \geq 0$  and $f, g \in L^2(\nu)$,
\begin{equation}\label{eq:adjoint}
\langle \PP_t f, g \rangle_{\nu} = \langle f, \PP^*_t g \rangle_{\nu}.
\end{equation}
We are now ready to state the following hypothesis.
\begin{assumption}\label{assmp1}
Let $\mathtt{N}\subseteq \N$ be a finite or a countable set, and for any $t \geq 0$, $(\Poly_n)_{n \in \mathtt{N}}$ (resp.~$(\Nu_n)_{n \in \mathtt{N}}$) be a set of eigenfunctions of $\PP_t$ (resp.~$\PP_t^*$) in $L^2(\nu)$ in the sense that there exist distinct $(\boldsymbol{\lambda}_n)_{n \in \mathtt{N}} \in \R_+$ such that for any $n \in \NN$ and $t \geq 0$, we have
\begin{eqnarray}
\PP_t \Poly_n &=& e^{-\boldsymbol{\lambda}_n t}\Poly_n,\label{eq:eigen} \\
\PP_t^* \Nu_n &=& e^{-\boldsymbol{\lambda}_n t}\Nu_n. \label{eq:coeigen}
\end{eqnarray}
\end{assumption}
We may also find convenient to characterize the $\Nu_n$'s by duality using \eqref{eq:adjoint}, i.e.~$\langle \PP_t f, \Nu_n \rangle_{\nu} = e^{-\boldsymbol{\lambda}_n t }\langle f, \Nu_n \rangle_{\nu}$, for all $f \in L^2(\nu)$. Note that the assumption on $(\boldsymbol{\lambda}_n)_{n \in \mathtt{N}}$ being of multiplicity $1$ is in fact for sake of simplicity since we mean to consider only one of the eigenfunctions in the eigenspace associated to each eigenvalue.\\
Next, without loss of generality, we assume that for any $n \in \mathtt{N}$,
\[ \langle \Poly_n, \Nu_n \rangle_{\nu} = 1 .\]
Indeed, if $\langle \Poly_n, \Nu_n \rangle_{\nu} = a_n \neq 0$ for $n \in \NN$, then we could consider the sequences $\bar{\Poly}_n = \frac{\Poly_n}{\sqrt{|a_n|}}$ and $\bar{\Nu}_n = \frac{\Nu_n}{\sqrt{|a_n|}}$, for which, obviously, we have $\langle \bar{\Poly}_n, \bar{\Nu}_n \rangle_{\nu} = 1$ for $n \in \NN$. We also note that the condition $\langle \Poly_n, \Nu_n \rangle_{\nu} = 1$ does not constrain the norms of the sequences $(\Poly_n)_{n \in \mathtt{N}}$ and $(\Nu_n)_{n \in \mathtt{N}}$ to be $1$, but it only follows from Cauchy-Schwartz inequality that, for any $n \in \mathtt{N}$,
\begin{equation*}
1 = \left| \langle \Poly_n, \Nu_n \rangle_{\nu} \right| \le \Vert \Poly_n \Vert_{\nu} \Vert \Nu_n \Vert_{\nu}.
\end{equation*}
In Lemma \ref{lemma:biorth} below, we shall show that $(\Poly_n,\Nu_n)_{n \in \NN}$ form a biorthogonal sequence in $L^2(\nu)$, i.e.~ $\langle \Poly_m, \Nu_n \rangle_{\nu} = \delta_{mn}$, where $\delta_{mn}$ is the Kronecker symbol defined in \eqref{eq:delta}. In particular, if $\PP$ is self-adjoint in $L^2(\nu)$, i.e.~for all $t \geq 0$, $\PP_t = \PP^*_t$, we have $\Poly_n = \Nu_n$ for $n \in \mathtt{N}$, and $(\Poly_n)_{n \in \mathtt{N}}$ form an orthonormal sequence in $L^2(\nu)$. Below we consider $\X$ to belong to one of the following three families of stochastic processes.

\subsubsection{Markov process}\label{subsec:Markov}
First, let $\X = X$ with $X=(X_t)_{t \geq 0}$ a Markov process defined on a filtered probability space $(\Omega, \F, (\F_t)_{t\geq0}, \Prob)$. We endow the state space $E$ with a sigma-algebra $\mathcal{E}$. Let its associated semigroup be the family of linear operators $\PP = P =(P_t)_{t \geq 0}$ defined, for any $t \geq 0$ and $f \in \mathcal{B}_b(E)$, by
\[ P_t f(x) = \E_x[f(X_t)]. \]
Next, we assume that for $t \geq 0$ and $f \in \mathcal{B}_b(E)$ the mapping $t \mapsto P_t f$ is continuous (this is equivalent to the stochastic continuity property of the process $X$), and the semigroup $P$ admits an invariant probability measure $\nu$, i.e.~$\nu P_t f = \nu f$. In such framework, a classical result states that the semigroup $P$ can be extended to a strongly continuous contraction semigroup in $L^2(\nu)$, see e.g.~Da Prato ~\cite{da2006introduction}, and by an abuse of notation, we still denote its extension to $L^2(\nu)$ by $P$. Note that the adjoint of $P$ in $L^2(\nu)$,  $P^*$ is the semigroup of a stochastic process which may not be necessarily a strong Markov one, but instead has the moderate Markov property, see e.g.~Chung and Walsh~\cite[Chapter 13]{chung_walsh} for more details.

\subsubsection{Bochner subordination}\label{subsec:Bochner}
In Section \ref{section:main_results} below, we also study the spectral projections correlation structure of subordinated Markov processes. \emph{Bochner subordination} is a transformation of a Markov process to a new one through random time change by an independent subordinator, i.e.~a real-valued L\'evy process with non-decreasing sample paths, see e.g.  \cite{bochner1949diffusion}, \cite{bochner2013harmonic}, \cite{schilling2012bernstein}. From the operator semigroup perspective, Bochner subordination is a classical method for generating a new semigroup of linear operators on a Banach space from an  existing one. More formally, using the notation of Section~\ref{subsec:Markov} above, for $P=(P_t)_{t \geq 0}$, a strongly continuous contraction semigroup in $L^2(\nu)$, and $(\mu_t)_{t \geq 0}$, a vaguely continuous convolution semigroup of probability measures on $[0, \infty)$, the subordination of $P$ in the sense of Bochner is defined by
\begin{equation}\label{subord}
P_t^{\varphi}f(x)= \int_{0}^{\infty} P_s f(x)\mu_t(ds), \quad t \geq 0,\ f \in \mathcal{B}_b(E).
\end{equation}
The superscript $\varphi$ alludes to the Laplace exponent of $(\mu_t)_{t \geq 0}$, which is a Bernstein function with the following representation, for $\lambda \geq 0$,
\begin{equation}\label{eq:varphi1}
\varphi(\lambda)=\varrho \lambda+\int_0^\infty(1-e^{-\lambda y})\vartheta(dy),
\end{equation}
where $\varrho \geq 0$, and $\vartheta$ is a L\'evy measure concentrated on $\R_+$ satisfying $\int_0^{\infty}(1 \wedge y)\vartheta(dy) <\infty$. Note that $(\mu_t)_{t \geq 0}$ gives rise to a L\'evy subordinator $\Tau=(\Tau_t)_{t\geq 0}$, which is assumed to be independent of $X$, and the law of $\Tau$ is uniquely characterized by its Laplace exponent $\varphi$, that is, for $t,\lambda \geq 0$,
\begin{equation}\label{eq:varphi}
\E \left[ e^{-\lambda \Tau_t} \right]= e^{-t\varphi(\lambda)}.
\end{equation}
We write $\X = X_{\Tau} =(X_{\Tau_t})_{t \geq 0}$ for the Markov process associated with the semigroup $\PP_t f(x) = P_t^{\varphi}f(x)= \E_x[f(X_{\Tau_t})]$. Moreover, one has that $\nu$ is also an invariant measure for the semigroup $P^{\varphi}$. Indeed, let $f \in \mathcal{B}_b(E)$ and assume, without loss of generality, $f$ is non-negative, then, for $t \geq 0$, we have
\begin{equation*}
\nu P_t^{\varphi} f = \langle P_t^{\varphi} f, \mathbbm{1} \rangle_{\nu} = \int_0^{\infty} \langle P_s f, \mathbbm{1} \rangle_{\nu} \mu_t(ds) =\int_0^{\infty} \nu f \mu_t(ds) = \nu f,
\end{equation*}
where we used Tonelli's theorem, the fact that $\nu$ is an invariant probability measure for $P$, and $\varphi(0)=0$ in \eqref{eq:varphi1}.
Therefore, as above, $P^{\varphi}$ can be extended to a contraction semigroup in $L^2(\nu)$. It is easy to note that the semigroup $P^{\varphi}$ shares the same eigenspaces and co-eigenspaces (eigenspaces for the adjoint) as $P$, and, in particular, we have the following.
\begin{proposition}\label{prop:bochner}
Let $(\Poly_n)_{n \in \NN}$ and $(\Nu_n)_{n \in \NN}$ be as defined in Assumption~\ref{assmp1} with $\PP = P$ of Section~\ref{subsec:Markov}. Then, $(\Poly_n)_{n \in \NN}$ and $(\Nu_n)_{n \in \NN}$ are the eigenfunctions of the semigroup $P^{\varphi}$ and its adjoint in $L^2(\nu)$, respectively, associated to the eigenvalues $(\varphi(\lambda_n))_{n \in \NN}$.
\end{proposition}

\begin{proof}
First, note that for $n \in \NN$, $\Poly_n \in L^2(\nu)$, and for any $t \geq 0$, we have
\begin{equation*}
P_t^{\varphi} \Poly_n = \int_0^{\infty}P_s \Poly_n\mu_t(ds) = \Poly_n \int_0^\infty e^{-\lambda_n s}\mu_t(ds)= e^{-t\varphi(\lambda_n)}\Poly_n,
\end{equation*}
where in the second equality we used \eqref{eq:eigen}, and the last step follows from \eqref{eq:varphi}. Next, for $f \in L^2(\nu)$, $n \in \NN$ and $t \geq 0$, note that
\begin{eqnarray*}
\langle P_t^{\varphi}f, \Nu_n \rangle_{\nu} &=& \int_E P_t^{\varphi}f (x) \Nu_n(x) \nu(dx)= \int_E \int_0^{\infty} P_s f(x)\mu_t(ds)\Nu_n(x)\nu(dx) \\
&=& \int_0^{\infty} \int_E P_s f(x)\Nu_n(x)\nu(dx)\mu_t(ds) = \int_0^{\infty} \langle P_s f, \Nu_n \rangle_{\nu} \mu_t(ds) \\
&=& \int_0^{\infty} \langle f, P_s^* \Nu_n \rangle_{\nu} \mu_t(ds) = \int_0^{\infty} e^{-\lambda_n s} \langle f, \Nu_n \rangle_{\nu} \mu_t(ds) = \langle f, \Nu_n \rangle_{\nu} e^{-t \varphi(\lambda_n)},
\end{eqnarray*}
where in the last two steps we used \eqref{eq:coeigen} and \eqref{eq:varphi}, and we were allowed to change the order of integration using Fubini's theorem, since by Cauchy-Schwartz inequality, we have
\begin{equation*}
\int_0^{\infty} \left\vert \langle P_s f, \Nu_n \rangle_{\nu} \right\vert \mu_t(ds) \le \int_0^{\infty} \Vert P_s f\Vert_{\nu} \Vert \Nu_n \Vert_{\nu} \mu_t(ds) \le \Vert f \Vert_{\nu} \Vert \Nu_n \Vert_{\nu} < \infty.
\end{equation*}
\end{proof}

\subsubsection{Non-Markovian processes obtained by a time-change with an inverse of a subordinator}
Let $\Tau$ denote the subordinator defined in \eqref{eq:varphi}, and define its right inverse, for $t>0$, by
\[ L_t=\inf \{s > 0; \Tau_s>t\}. \]
We point out that $t \mapsto \Tau_t$ is right-continuous and non-decreasing, and hence $t \mapsto L_t$ is also right-continuous and non-decreasing. In particular, when $t \mapsto \Tau_t$ is a.s.~increasing, which is equivalent to $\varphi(\infty)=\infty$ in \eqref{eq:varphi1}, then $t \mapsto L_t$ is continuous and $L_{\Tau_t}=t$ a.s., whereas $\Tau_{L_t}>t$ a.s. Next, let $l_t$ denote the distribution of $L_t$, i.e.~for any $B$ Borelian set of $\R_+$,  $l_t(B)=\Prob(L_t \in B)$. Then, for any $\lambda \geq 0$ and $t \geq 0$, its Laplace transform is denoted by
\begin{equation}
\eta_{t}(\lambda) = \int_0^{\infty}e^{-\lambda s}l_t(ds).
\end{equation}
For sake of simplicity, we assume that $\Prob(L_t < \infty) = \eta_t(0) = \int_0^{\infty} l_t(ds) = 1$ for all $t \geq 0$. However, all of the results presented below could be easily adapted to the case when $\int_0^{\infty} l_t(ds) < 1$ for some $t \geq 0$ (and hence, all $t \geq 0$). Let $\PP = P^{\eta}=(P^{\eta}_t)_{t\geq 0}$ be the family of linear operators defined, for $f \in \mathcal{B}_b(E)$ and $t \geq 0$, by
\[ P^{\eta}_t f(x) = \int_0^\infty P_s f(x) l_t(ds). \]
The corresponding time-changed process will be denoted by $\X =X_L = (X_{L_t})_{t \geq 0}$. As mentioned in the introduction above, this time-change with an inverse of a subordinator in specific situations was discussed in~\cite{leonenko2013correlation} and ~\cite{mijena2014corr}.
In the following we provide some basic properties of $P^{\eta}$.
\begin{proposition}
For any $f \in \mathcal{B}_b(E)$ and $t \geq 0$, $\nu P^{\eta}_t f = \nu f$, i.e.~$\nu$ is a marginal stationary measure, and it is also a limiting distribution for $P^{\eta}$, i.e.~$\lim_{t \rightarrow \infty} \nu P^{\eta}_t f = \nu f$. Moreover, for all $t \geq 0$, $P_t^{\eta}$ can be extended to a contraction in $L^2(\nu)$.
\end{proposition}
\begin{proof}
Let $f \in \mathcal{B}_b(E)$ and non-negative, then, for any $t \geq 0$, we have, as above,
\begin{equation*}
\nu P_t^{\eta} f = \langle P_t^{\eta} f, \mathbbm{1} \rangle_{\nu} = \int_0^{\infty} \langle P_s f, \mathbbm{1} \rangle_{\nu} \ l_t(ds) = \nu f,
\end{equation*}
where we used Tonelli's theorem, the fact that $\nu$ is an invariant measure for $P$ and $\int_0^{\infty} l_t(ds) = 1$. Next, for a fixed $t \geq 0$ and any $f \in L^2(\nu)$, we note that
\begin{eqnarray*}
\Vert P_t^{\eta} f \Vert_{\nu}^2 = \int_0^{\infty} \left( P_t^{\eta} f(x) \right)^2 \nu(dx) &=&\int_0^{\infty} \left(\int_0^{\infty}  P_s f(x)l_t(ds) \right)^2  \nu(dx) \\
& \le &  \int_0^{\infty} \int_0^{\infty}  (P_s f(x))^2 \nu(dx) l_t(ds)\\
&=& \int_0^{\infty} \Vert P_s f \Vert_{\nu}^2 l_t(ds) \le \Vert f \Vert_{\nu}^2,
\end{eqnarray*}
where we used Jensen's inequality and Tonelli's theorem, and in the last step we used the fact that $P$ is a contraction semigroup in $L^2(\nu)$, and that the total mass of $l_t$ is $1$.
\end{proof}
\subsection{Covariance and correlation functions}~
Notions of covariance and correlation functions have been intensively studied in the statistical literature. For example, the introduction of  \emph{distance covariance} and \emph{distance correlation}, which are analogous to product-moment covariance and correlation but generalize and extend these classical bivariate measures of dependence, is well detailed in Sz\'ekely et~al.~\cite{distance_corr}. More formally, let $X$ and $Y$ be two random vectors with finite first moments in $\R^p$ and $\R^q$, $p,q \in \N$, respectively. For any $d\in \N$, $|\cdot|_d$ denotes the Euclidean norm of the vector in $\R^d$, and
\[c_d = \frac{\pi^{(1+d)/2}}{\Gamma((1+d)/2)}. \]
Then, the \emph{distance covariance} between random vectors $X$ and $Y$ is the non-negative number $V(X,Y)$ defined by
\[ V^2(X,Y) = \Vert f_{X,Y} - f_X f_Y \Vert^2 = \frac{1}{c_p c_q}\int_{\R^{p+q}} \frac{|f_{X,Y}(t,s) - f_X(t)f_Y(s)|^2}{|t|^{1+p}_p |s|^{1+q}_q}dtds, \]
where $f_X$ and $f_Y$ are the characteristic functions of the random vectors $X$ and $Y$, respectively and $f_{X,Y}$ denotes their joint characteristic function. Similarly, the \emph{distance correlation coefficient} between random vectors $X$ and $Y$ with finite first moments is the nonnegative number $R(X,Y)$ defined by
\begin{equation*}
R(X,Y)=\begin{cases}
               \frac{V^2(X,Y)}{\sqrt{V^2(X,X)V^2(Y,Y)}}, \quad \text{if } V^2(X,X)V^2(Y,Y)>0,\\
               0, \quad \text{if } V^2(X,X)V^2(Y,Y)=0.
            \end{cases}
\end{equation*}
Furthermore, note that $R \in [0,1]$, and $R(X,Y) \le |\rho(X,Y)|$, where $\rho$ denotes the Pearson correlation coefficient, and equality holds when $\rho = \pm 1$. We remark that distance correlation measures the strength of relation between $X$ and $Y$, and it generalizes the idea of correlation in two fundamental ways:
\begin{itemize}
\item[(i)] $R(X,Y)$ is defined for $X$ and $Y$ in arbitrary dimensions;
\item[(ii)] $R(X,Y) = 0$ characterizes independence of $X$ and $Y$.
\end{itemize}
The distance correlation coefficient is especially useful for complicated dependence structures in multivariate data. Sz\'ekely et~al.~\cite{distance_corr} discuss some asymptotic properties and present implementation of the independence test and Monte Carlo results. It is worth to mention that Sz\'ekely and Rizzo~\cite{brownian_cov} introduce the notion of covariance with respect to a stochastic process and show that population distance covariance coincides with the covariance with respect to Brownian motion. Furthermore, Bhattacharjee~\cite{dist_corr_stat} elaborates the application of a Bayesian approach in distance correlation which can be useful to test the linear relation between variables.

Another interesting measure of dependence between two random variables $X$ and $Y$ is the \emph{maximal correlation coefficient} introduced by Gebelein~\cite{gebelein} and later studied by R\'{e}nyi~\cite{renyi}, Papadatos and Xifara~\cite{papadatos_maxcorr}, Beigi and Gohari~\cite{beigi_maxcorr}, among other authors. It is defined as
\begin{equation}\label{eq:max_corr}
\rho_{\max}(X,Y) = \sup_{f,g} \lbrace \rho(f(X),g(Y));\  0<\E |f(X)|^2<\infty, 0<\E |g(Y)|^2<\infty \rbrace,
\end{equation}
where the supremum is taken over all Borel measurable functions $f,g :\R \rightarrow \R$, and  $\rho(X,Y)$ is the classical (Pearson) correlation coefficient between the random variables $X$ and $Y$. Definition \eqref{eq:max_corr} is equivalent to
\begin{equation*}
\rho_{\max}(X,Y) = \sup_{f,g}\lbrace \E[f(X)g(Y)]; \E[f(X)] = \E[g(Y)] = 0, \E|f(X)|^2 = \E|g(Y)|^2 = 1 \rbrace,
\end{equation*}
where the supremum is again taken over Borel measurable functions $f,g :\R \rightarrow \R$. The main role of $\rho_{\max}(X,Y)$ is that of a convenient numerical measure of dependence between $X$ and $Y$. In particular, it has the tensorization property, i.e.~it is unchanged when computed
for i.i.d.~copies. Furthermore, $\rho_{\max}(X,Y) = 0$ if and only if $X$ and $Y$ are independent. Even though the maximal correlation coefficient plays a fundamental role in various areas of statistics, despite its usefulness, it is often difficult to calculate it in an explicit form, except in some rare cases. Some well-known exceptions are provided by the results of \cite{gebelein}, \cite{lancaster}, \cite{dembo}, \cite{bryc}, \cite{yu}.

Now, let $\X$ be a stochastic process, and $\nu$ be a Radon measure on the state space of $\X$. We define the \emph{covariance and correlation functions under $\nu$} in the following way. Let $s, t \geq 0$, then for any functions $f,g \in L^2(\nu)$,
\begin{equation}\label{eq:def_cov}
\C_{\nu} (f(\X_t), g(\X_s)) = \E_{\nu} [f(\X_t)g(\X_s)] - \E_{\nu}[f(\X_t)]\E_{\nu} [g(\X_s)],
\end{equation}
\begin{equation}\label{eq:def_corr}
\rho_{\nu}(f(\X_t), g(\X_s)) = \begin{cases} \frac{\C_{\nu}(f(\X_t),g(\X_s)) }{std_{\nu}(f(\X_t))std_{\nu}(g(\X_s))}, \quad \text{if } std_{\nu}(f(\X_t))std_{\nu}(g(\X_s)) > 0 , \\
0, \quad \text{if } std_{\nu}(f(\X_t))std_{\nu}(g(\X_s)) = 0,
\end{cases}
\end{equation}
where $std_{\nu}$ stands for the standard deviation defined by
\[ std_{\nu}(f(\X_t)) = \sqrt{\C_{\nu}(f(\X_t), f(\X_t))}.\]
\begin{definition}\label{definition}
When $\nu$ is a marginal stationary measure for $\X$ and Assumption~\ref{assmp1} holds, for $m,n \in \NN$ and $t,s >0$, we call $\rho_{\nu}(\Poly_m(\X_t), \Poly_n(\X_s))$ (resp. $\rho_{\nu}(\Poly_m(\X_t), \Nu_n(\X_s))$) \emph{(resp. biorthogonal) spectral projections correlation functions}.
\end{definition}
The rest of the paper is organized as follows. In Section \ref{section:main_results}, we present the main results which include explicit expressions for the spectral projections correlation structure of non-reversible Markov processes, of their subordinated counterparts, as well as of non-Markovian processes, obtained by time-changing a Markov process with an inverse of a subordinator. In Section \ref{section:examples}, we illustrate our results for the class of generalized Laguerre processes, which are associated with non-self-adjoint and non-local semigroups. The proofs of the main results are presented in Section \ref{section:proofs}.

\section{Main results}\label{section:main_results}
Let us start with $\X = X$ a Markov process admitting an invariant probability measure $\nu$, i.e.~$\nu P_t f = \nu f$ for all $t \geq 0$ and $f \in L^2(\nu)$ where $P$ is the $L^2(\nu)$-semigroup. Recall from Assumption \ref{assmp1} that $\NN \subset \N$ is a finite or a countable set, and for any $t \geq 0$, $(\Poly_n)_{n \in \NN}$ and $(\Nu_n)_{n \in \NN}$ denote the sets of eigenfunctions of $P_t$ and $P_t^*$, respectively. Next, for $m,n \in \N$, let $\delta_{mn}$ be the Kronecker symbol, i.e.
\begin{equation}\label{eq:delta}
\delta_{mn}=\begin{cases}
               0, \quad \text{if } m \neq n,\\
               1, \quad \text{if } m = n.
            \end{cases}
\end{equation}
Then, we have the following characterization of the (biorthogonal) spectral projections correlation functions.
\begin{theorem}\label{thm:corr}
Let $m,n \in \mathtt{N}$. Then, for any $t \geq s > 0$, we have
\begin{equation*}
\rho_{\nu}(\Poly_m(X_t), \Nu_n(X_s)) = e^{-\lambda_m(t-s)}\kappa_{\nu}^{-1}(m) \delta_{mn},
\end{equation*}
and
\begin{equation*}
\rho_{\nu}(\Poly_m(X_t), \Poly_n(X_s)) = e^{-\lambda_m(t-s)} c_{\nu}(n,m),
\end{equation*}
where $\kappa_{\nu}(m) =\Vert \Poly_m \Vert_{\nu} \Vert \Nu_m \Vert_{\nu}$ and $-1 \le c_{\nu}(n,m) = \frac{\langle \Poly_n, \Poly_m \rangle_{\nu}}{\Vert \Poly_n \Vert_{\nu} \Vert \Poly_m \Vert_{\nu}} \le 1$. Consequently,  $c_{\nu}(n,n)=1$ for any $n \in \mathtt{N}$.
\end{theorem}



\begin{remark}\label{rem:cond}
We shall show in Lemma~\ref{lemma:biorth} below that $(\Poly_n, \Nu_n)_{n \in \mathtt{N}}$ form a biorthogonal sequence in $L^2(\nu)$ in the sense that $\langle \Poly_m, \Nu_n \rangle_{\nu} = \delta_{mn}$ for any $m,n \in \mathtt{N}$. Then, each (non-orthogonal) spectral projection is given by
\begin{equation*}
\boldsymbol{\mathcal{P}}_m f = \langle f, \Poly_m \rangle_{\nu} \Nu_m, \quad \text{for } f \in L^2(\nu).
\end{equation*}
Moreover, in this context, the number
\[ \kappa_{\nu}(m) = \Vert \Poly_m \Vert_{\nu} \Vert \Nu_m \Vert_{\nu} \]
is called the \emph{condition number} of the eigenvalue $\lambda_m$ and corresponds to the norm of the operator $\boldsymbol{\mathcal{P}}_m$, see e.g.~Davies \cite{davies_cond_number}. The condition number measures how unstable the eigenvalues are under small perturbations of the operator $P_t$. We note that when $(\Poly_n, \Nu_n)_{n \in \NN}$ form an orthonormal sequence, then $\kappa_{\nu}(m)=1$.
\end{remark}
\begin{remark}\label{rem:tame}
Recall that a biorthogonal system $(\Poly_n, \Nu_n)_{n \in \mathtt{N}}$ is called \emph{tame} in $L^2(\nu)$ if $\NN = \infty$, it is complete (i.e.~$\overline{\text{Span}}(\Poly_n)_{n \in \NN}=L^2(\nu)$) and
\[\kappa_{\nu}(m)  = O (m^{\beta}), \]
for all $m \in \NN$ and some $\beta$, i.e.~there exists $b \in \R_+$ and $m_0 \in \N$ such that $|\kappa_{\nu}(m)| = \kappa_{\nu}(m) \le b m^{\beta}$ for all $m \geq m_0$, see Davies \cite{davies_cond_number}. Otherwise, we say that the system is \emph{wild}. It is easy to note that if $(\Poly_n)_{n \in \NN}$ is a basis in $L^2(\nu)$, then $\kappa_{\nu}(m)$ is uniformly bounded, so the system is tame with $\beta = 0$.
\end{remark}
\begin{remark}\label{rem:k-depart}
When $P=(P_t)_{t \geq 0}$ is a self-adjoint compact semigroup, then $\NN=\N$ and $(\Poly_n)_{n\in \N}=(\Nu_n)_{n\in \N}$ form an orthonormal basis of $L^2(\nu)$. However, when $P$ is non-self-adjoint, then $(\Poly_n, \Nu_n)_{n \in \NN}$ do not form, in general, a basis of $L^2(\nu)$. A necessary condition for $(\Poly_n, \Nu_n)_{n \in \NN}$ to form a basis is that the condition number $\kappa_{\nu}(m)$ is uniformly bounded. In this sense, the rate of growth of $\kappa_{\nu}(m)$ also can be seen as a measure of departure of these sequences from the basis property.
\end{remark}
\begin{remark}
From the definition of the inner product, we note that $c_{\nu}(n,m) = \cos \measuredangle(\Poly_n,\Poly_m)$ and $\arccos c_{\nu}(n,m)$ measures the angle between the polynomials $\Poly_n$ and $\Poly_m$ denoted by $\measuredangle(\Poly_n,\Poly_m)$. In particular, the sequence $(\Poly_n)_{n \geq 0}$ is orthogonal if and only if $c_{\nu}(n,m)=0$ for $n \neq m$.
\end{remark}
The proof of Theorem~ \ref{thm:corr} is presented in Section~ \ref{sec:proof_corr}.
\begin{lemma}\label{lemma:corr}
For any $f,g \in L^2(\nu)$ and $t \geq 0$,
\begin{equation*}
\rho_{\nu}(f(X_t),g(X_t)) = \frac{\langle f,g \rangle_{\nu} - \nu f \cdot \nu g}{\sqrt{\nu f^2 - (\nu f)^2} \cdot \sqrt{\nu g^2 - (\nu g)^2}}.
\end{equation*}
In particular, for any $m,n \in \NN$ and $t \geq 0$,
\begin{eqnarray*}
\rho_{\nu}(\Poly_m(X_t),\Poly_n(X_t)) &=& c_{\nu}(n,m),\\
\rho_{\nu}(\Poly_m(X_t),\Nu_n(X_t)) &=& \kappa_{\nu}^{-1}(m)\delta_{mn}.
\end{eqnarray*}
\end{lemma}
\begin{remark} \label{lemma:nf=0}
Note that if $f,g \in L^2(\nu)$ are such that $\nu f = \nu g = 0$, then for any $t \geq 0$
\begin{equation*}
\rho_{\nu}(f(X_t),g(X_t)) = \frac{\langle f,g \rangle_{\nu}}{\Vert f \Vert_{\nu} \Vert g \Vert_{\nu}}.
\end{equation*}
\end{remark}

The proof of Lemma~\ref{lemma:corr} is presented in Section~\ref{sec:proof_lemma_corr}.

We now proceed by studying the effect of the stochastic time-change in the analysis of the spectral projections correlation function. First, we start with Bochner subordination. To this end, recall that $\Tau=(\Tau_t)_{t\geq 0}$ is a subordinator with Laplace exponent $\varphi$ and transition kernel $\mu_t(ds)$, i.e.~
$$\E \left[ e^{-\lambda \Tau_t} \right]= \int_{0}^{\infty}e^{-\lambda s}\mu_t(ds)=e^{-t\varphi(\lambda)}, \quad \lambda >0, \ t \geq 0,$$
where $\varphi(\lambda)=\varrho \lambda+\int_0^\infty(1-e^{-\lambda y})\vartheta(dy)$ with $\varrho \geq 0$, and $\vartheta$ being a L\'evy measure satisfying $\int_0^{\infty}(1 \wedge y)\vartheta(dy) <\infty$. Denote the semigroup of the subordinated process by $P^{\varphi} = (P_t^{\varphi})_{t \geq 0}$, i.e.~for $f \in \mathcal{B}_b(E)$ and $t \geq 0$,
\[P_t^{\varphi}f(x)= \E_x[f(X_{\Tau_t})]. \]
We recall from Section~\ref{sec:prelim} that $P^{\varphi}$ defines an $L^2(\nu)$-Markov semigroup with $\nu$ as an invariant measure. By combining Proposition~\ref{prop:bochner} and Theorem~\ref{thm:corr}, we obtain the following characterization of the spectral projections correlation structure of the subordinated process.
\begin{corollary}\label{cor:corr_sub}
Moreover, for $m,n \in \NN$ and $t \geq s > 0$, we have
\begin{equation}
\rho_{\nu}(\Poly_m(X_{\Tau_t}), \Nu_n(X_{\Tau_s}))= e^{-\varphi(\lambda_m) (t-s)}\kappa_{\nu}^{-1}(m) \delta_{mn},
\end{equation}
and
\begin{equation}
\rho_{\nu}(\Poly_m(X_{\Tau_t}), \Poly_n(X_{\Tau_s}))= e^{-\varphi(\lambda_m) (t-s)}c_{\nu}(n,m).
\end{equation}
\end{corollary}
\begin{remark}
Since $(\Poly_n,\Nu_n)_{n \in \NN}$ form a biorthogonal sequence in $L^2(\nu)$, and are, respectively, the eigenfunctions of $P^{\varphi}$ and its adjoint in $L^2(\nu)$, the correlation function $\rho_{\nu}(\Poly_m(X_{\Tau_t}), \Poly_n(X_{\Tau_s}))$ (resp. $\rho_{\nu}(\Poly_m(X_{\Tau_t}), \Nu_n(X_{\Tau_s}))$) is the (resp. biorthogonal) spectral projections correlation function of the process $(X_{\Tau_t})_{t \geq 0}$.
\end{remark}
We continue with another stochastic time-change given by an inverse of a subordinator, which, as explained in Section \ref{sec:prelim}, gives rise to a non-Markovian process. Recall that the inverse of the subordinator $\Tau$ is defined for $t>0$ by $L_t=\inf \{s > 0; \Tau_s>t\}$, its distribution is denoted by $l_t$, and its Laplace transform by $\eta_t$, that is for any $\lambda>0$,
\begin{equation*}
\eta_{t}(\lambda) = \int_0^{\infty}e^{-\lambda s}l_t(ds).
\end{equation*}
Also recall that we assume $\eta_t(0) = \int_0^{\infty} l_t(ds) = 1$ for all $t \geq 0$. Then, $P^{\eta}=(P^{\eta}_t)_{t\geq 0}$, defined, for $t \geq 0$ and $f \in L^2(\nu)$, by
\[P^{\eta}_t f(x) = \int_0^\infty P_s f(x) l_t(ds), \]
is a linear operator, and the corresponding time-changed process will be denoted by $X_L = (X_{L_t})_{t \geq 0}$. Note that Leonenko~et~al.~\cite{leonenko2013correlation} and Mijena and Nane~\cite{mijena2014corr} characterize the correlation structure of so-called Pearson diffusions when they are time-changed by an inverse of a linear combination of independent stable subordinators. We extend their methodology by first considering a general Markov process with biorthogonal spectral projections, and then time-changing it with an inverse of any independent subordinator.
We also point out that by following a line of reasoning similar to the proof of Proposition~\ref{prop:bochner}, it can be shown that the biorthogonal sequence $(\Poly_n,\Nu_n)_{n \in \NN}$ represent a set of eigenfunctions of the linear operator $P^{\eta}_t$, $t \geq 0$ and its adjoint in $L^2(\nu)$, respectively. Thus, $\rho_{\nu}(\Poly_m(X_{L_t}), \Poly_n(X_{L_s}))$ (resp.~$\rho_{\nu}(\Poly_m(X_{L_t}), \Nu_n(X_{L_s}))$) is the (resp.~biorthogonal) spectral projections correlation function of the process $X_L$. Finally, we set the following notation.
\begin{itemize}
\item[(a)] We write $f \stackrel{a}{\sim} g$ for $a \in [0,\infty]$ if $\displaystyle{\lim_{x \rightarrow a}\frac{f(x)}{g(x)}=1}$. We may write $f(x) \stackrel{x \rightarrow a}{\sim} g(x)$ to emphasize dependency on the variable $x$.
\begin{itemize}
\item[(a1)] $f$ is called a long-tailed function if $\tau_y f(x) \stackrel{x \rightarrow \infty}{\sim} f(x)$ for any fixed $y>0$, where $\tau_y f(x)=f(x+y)$ is the shift operator.
\item[(a2)] $f$ is called slowly varying at $0$ if $d_a f(x) \stackrel{x \rightarrow 0}{\sim} f(x) $ for any fixed $a > 0$, where $d_a f(x) = f(ax)$ is the dilation operator.
\item[(a3)] We say that $f$ is strongly regularly varying at $a$ with index $0<\alpha < 1$ if $f \stackrel{a}{\sim} p_{\alpha}$, where $p_{\alpha}(x) = C x^{\alpha}$ for some constant $C>0$.
\end{itemize}
\item[(b)] We write $f \stackrel{a}{\asymp} g$ if there exists a constant  $C>0$ such that $\frac{1}{C} g(x) \le f(x) \le C g(x)$ for $x \geq a$.
\item[(c)] We write  $f=O(g)$ if $\overline{\lim}_{x \rightarrow \infty}\left| \frac{f(x)}{g(x)}\right| < \infty$.
\end{itemize}

We are now ready to state our last main result.
\begin{theorem}\label{thm:t-ch}
Let $m,n \in \mathtt{N}$. Then, for $t \geq s > 0$,
\begin{equation}\label{eq:corr_t_ch}
\rho_{\nu}(\Poly_m(X_{L_t}), \Poly_n(X_{L_s})) =c_{\nu}(n,m)\Bigl( \lambda_m \int_0^{s}\eta_{t-r}(\lambda_m)U(dr) + \eta_{t}(\lambda_m) \Bigr),
\end{equation}
and
\begin{equation}\label{eq:corr_t_ch_nu}
\rho_{\nu}(\Poly_m(X_{L_t}), \Nu_n(X_{L_s})) = \kappa_{\nu}^{-1}(m)\delta_{mn} \Bigl( \lambda_m \int_0^{s}\eta_{t-r}(\lambda_m)U(dr) + \eta_{t}(\lambda_m) \Bigr),
\end{equation}
where $U(dr) = \int_0^{\infty} \Prob(\Tau_t \in dr)dt$ is the renewal measure of the subordinator $\Tau$.
Moreover, for any fixed $s > 0$,
\begin{eqnarray*}
c_{\nu}(n,m) \eta_t(\lambda_m)(\lambda_m \E[L_s]+1) \le &\rho_{\nu}(\Poly_m(X_{L_t}), \Poly_n(X_{L_s}))& \le c_{\nu}(n,m) \eta_{t-s}(\lambda_m)(\lambda_m \E[L_s]+1) ,\\
\kappa_{\nu}^{-1}(m)\delta_{mn} \eta_t(\lambda_m)(\lambda_m \E[L_s]+1) \le &\rho_{\nu}(\Poly_m(X_{L_t}), \Nu_n(X_{L_s}))& \le \kappa_{\nu}^{-1}(m)\delta_{mn} \eta_{t-s}(\lambda_m)(\lambda_m \E[L_s]+1).
\end{eqnarray*}
Furthermore, if for a fixed $s > 0$, $\overline{\lim}_{t \rightarrow \infty}\frac{\eta_{t-s}(\lambda_m)}{\eta_t(\lambda_m)} =C$ for some constant $C=C(s,\lambda_m)(\lambda_m \E[L_s]+1)$, then there exists $t_0 >0$ such that
\begin{eqnarray*}
\rho_{\nu}(\Poly_m(X_{L_t}), \Poly_n(X_{L_s})) &\stackrel{t_0}{\asymp} & c_{\nu}(n,m) \eta_{t}(\lambda_m)(\lambda_m \E[L_s]+1) ,\\
\rho_{\nu}(\Poly_m(X_{L_t}), \Nu_n(X_{L_s}))& \stackrel{t_0}{\asymp} & \kappa_{\nu}^{-1}(m)\delta_{mn} \eta_{t}(\lambda_m)(\lambda_m \E[L_s]+1).
\end{eqnarray*}
In particular, if $t \mapsto \eta_t(\lambda_m)$ is a long-tailed function, we have
\begin{eqnarray}
\rho_{\nu}(\Poly_m(X_{L_t}), \Poly_n(X_{L_s})) &\stackrel{t \rightarrow \infty}{\sim}& c_{\nu}(n,m)\eta_{t}(\lambda_m)(\lambda_m \E[L_s]+1), \label{eq:assympt} \\
\rho_{\nu}(\Poly_m(X_{L_t}), \Nu_n(X_{L_s})) &\stackrel{t \rightarrow \infty}{\sim}&  \kappa_{\nu}^{-1}(m)\delta_{mn}\eta_{t}(\lambda_m)(\lambda_m \E[L_s]+1).\label{eq:assympt_bi}
\end{eqnarray}
\end{theorem}
The proof of this theorem is presented in Section \ref{sec:proof_t_ch}. We complete this part with the following result which provides a sufficient condition for $\eta_t$ to be a long-tailed function.
\begin{proposition}\label{prop:long_tailed}
For any $\lambda > 0$, there exists a positive random variable $X_{\lambda}$ such that $\eta_t(\lambda)$ is the tail of its distribution, i.e.~$\eta_t(\lambda) = \Prob(X_{\lambda}>t)$, $t >0$. Moreover, $\eta_t(\lambda)$ is a long-tailed distribution  if $\varphi$ is strongly regularly varying at $0$.
\end{proposition}
The proof of this proposition is presented in Section~\ref{sec:proof_long_tailed}.

\subsection{Interpretation of the (biorthogonal) spectral projections correlation functions for statistical properties} The results presented above regarding the (biorthogonal) spectral projections correlation functions and their asymptotic behavior provide an interesting approach for designing statistical tests in order to identify substantial properties of a stochastic process. More formally, we start by assuming that the sample $\widehat{\X} = (\widehat{\X}_1,\cdots,\widehat{\X}_T)$, with $T \in \N$ large, is coming from a stochastic process $\X$ which belongs to some family with a marginal stationary measure $(\nu_i)_{i \in I}$ and associated biorthogonal sequence $\left( (\Poly_n^{(i)}, \Nu_n^{(i)})_{n \in \NN} \right)_{i \in I}$ as defined in Assumption~\ref{assmp1}, where $I$ is the index set of the family. For example, in the case when $\X$ belongs to the family of generalized Laguerre processes presented in Section~\ref{section:gL} below, we can consider one element from each of the following sub-families: a pure diffusion, a diffusion component and jumps with finite activity, a diffusion component and jumps with infinite activity, and a pure jump process.
Now, based on the (biorthogonal) spectral projections correlation structure, one can identify
\begin{itemize}
\item[(a)] how far from symmetry (self-adjointness) the process is,
\item[(b)] what type of range dependence (short-to-long) it displays, and
\item[(c)] the path properties of the process (c\'adl\'ag or continuous paths).
\end{itemize}
For designing statistical tests, one can rely on the estimates of $\kappa_{\nu_i}(m)$ and/or $c_{\nu_i}(n,m)$, $i \in I$, $n,m \in \NN$. Since $\left(\kappa_{\nu_i}(m)\right)_{m \in \NN}$ contain information about both of the sequences $(\Poly_n^{(i)})_{n \in \NN}$ and $(\Nu_n^{(i)})_{n \in \NN}$, below we describe some statistical tests involving the condition number. However, the estimates of $c_{\nu_i}(n,m)$ can be useful to further refine the search of the process. More precisely, based on the main results presented in Section~\ref{section:main_results}, one can make the following implications.
\begin{itemize}
\item[(a)] To study the possible \emph{departure from symmetry} of $\widehat{\X}$, see Remark ~\ref{rem:k-depart}, following the results provided by Lemma~\ref{lemma:corr}, we first take $t=s = k$ for some $k \in \{1,\cdots, T \}$ and $m = n \in \NN$. Then, since the marginal stationary measure guarantees that the statistical properties of the process do not change over time, for each $i \in I$, we compute the empirical estimates of the condition number $\kappa_{\nu_i}(m)$ for some $m \in \NN$, by
\begin{eqnarray}\label{eq:kappa_hat}
\widehat{\kappa}_{\nu_i}^{-1}(m) &=&  \widehat{\rho}_{\nu_i}(\Poly_m^{(i)}(\widehat{\X}_k), \Nu_m^{(i)}(\widehat{\X}_k)) \nonumber \\
& =&
\frac{\sum_{j=1}^T \left( \Poly_m^{(i)}(\widehat{\X}_j)-\overline{\Poly}_m^{(i)}(\widehat{\X}) \right) \left( \Nu_m^{(i)}(\widehat{\X}_j)-\overline{\Nu}_m^{(i)}(\widehat{\X})\right)}{\sqrt{\sum_{j=1}^T\left( \Poly_m^{(i)}(\widehat{\X}_j)-\overline{\Poly}_m^{(i)}(\widehat{\X}) \right)^2}\sqrt{\sum_{j=1}^T\left( \Nu_m^{(i)}(\widehat{\X}_j)-\overline{\Nu}_m^{(i)}(\widehat{\X})\right)^2}} ,
\end{eqnarray}
where $\overline{\Poly}_m^{(i)}(\widehat{\X})=\frac{1}{T}\sum_{j=1}^T \Poly_m^{(i)}(\widehat{\X}_j)$ and $\overline{\Nu}_m^{(i)}(\widehat{\X})=\frac{1}{T}\sum_{j=1}^T \Nu_m^{(i)}(\widehat{\X}_j)$ are the sample means.
Next, we compute the theoretical condition number by
\[ \kappa_{\nu_i}(m) = \Vert \Poly_m^{(i)}\Vert_{\nu_i}\Vert \Nu_m^{(i)}\Vert_{\nu_i}. \]
Finally, to identify the couple $(\Poly_n^{(i)}, \Nu_n^{(i)})_{n \in \NN}$, we choose $\epsilon_{S} > 0$, and check if
\begin{equation}\label{eq:eps}
| \kappa_{\nu_i}(m) - \widehat{\kappa}_{\nu_i}(m)| < \epsilon_{S}.
\end{equation}
For the next step, for sake of simplicity, we suppose that there is only one $\bar{i} \in I$ such that the condition~\eqref{eq:eps} is satisfied.
\item[(b)] To asses the \emph{range dependence} of the sample, we study the asymptotic behavior of the empirical correlation $\widehat{\rho}_{\nu_{\bar{i}}}(\Poly_m^{(\bar{i})}(\widehat{\X}_k), \Nu_m^{(\bar{i})}(\widehat{\X}_j))$, $k,j \in \{1,\cdots,T\}$, $k > j$, $m \in \NN$. More formally, we first compute $\widehat{\kappa}_{\nu_{\bar{i}}}(m)$ by \eqref{eq:kappa_hat}, fix some $j \in \{1,\cdots,T\}$ (one can simply set $j = 1$ or $j = 2$), and we proceed by studying
\begin{equation}
g_{\lambda_m}(k) = \widehat{\kappa}_{\nu_{\bar{i}}}(m)\cdot \widehat{\rho}_{\nu_{\bar{i}}}(\Poly_m^{(\bar{i})}(\widehat{\X}_k), \Nu_m^{(\bar{i})}(\widehat{\X}_j)).
\end{equation}
Now, if $k \mapsto g_{\lambda_m}(k)$, $j < k \in \{1,\cdots,T\}$ exhibits exponential decay with respect to $\lambda_m$, then we have short-range dependence. In contrast, if it exhibits a polynomial decay, then the process has long-range dependence, and, in particular, it is not a (subordinated) Markov process. We remark that although these two cases are the most popular ones discussed in the literature, depending on the rate of  decay of the correlation function, the process can exhibit short-to-long-range dependence. The concept of long-range dependence has been repeatedly used to describe properties of financial time series such as stock prices, foreign exchange rates, market indices and commodity prices. In this context, based on the behavior of (biortogonal) spectral projections correlation functions, in their working paper~\cite{empirical_corr}, the authors provide a more detailed empirical study to detect the (short-to-long-) range dependence in volatility in financial markets.
\item[(c)] Finally, to study the \emph{path properties} of the process, i.e.~the presence of jumps and their activity, we study the behavior of $\widehat{\kappa}_{\nu_{\bar{i}}}(m)$ for large $m$. To illustrate this with a specific example, let us consider the class of generalized Laguerre processes introduced in Section ~\ref{section:gL}. Note that this class encompasses a range of symmetries and jumps. Then, one can identify the following cases.
\begin{enumerate}[(i)]
\item If $\widehat{\kappa}_{\nu_{\bar{i}}}(m) = 1$, $m \in \NN$, then the process is a pure diffusion, see Section~ \ref{section:self_adj}.
\item If $\widehat{\kappa}_{\nu_{\bar{i}}}(m)=O(m^{\beta})$ for some $\beta$, then the process has both a diffusion component and a jump component with finite activity, see Section~ \ref{section:small_pert}.
\item If $\widehat{\kappa}_{\nu_{\bar{i}}}(m)=O(e^{\epsilon m})$ for any $\epsilon >0$, then, similarly, the process has both diffusion and jump components while in this case jumps have infinite activity, see Section~ \ref{section:big_jumps}.
\item If $\widehat{\kappa}_{\nu_{\bar{i}}}(m)=O \left(e^{m^{\beta}}\right)$ for some $\beta$, then the process is a pure jump process.
\end{enumerate}
The problem of deciding whether the continuous-time process which models an economic or financial time series has continuous paths or exhibits jumps is an important issue. For example, A\"{i}t-Sahalia and Jacod~\cite{aitsahalia_jacod_jumps} design a test to identify the presence of jumps in a discretely observed semimartingale, based on power variations sampled at different frequencies. Furthermore, in this setting, the authors of~\cite{aitsahalia_jacod_jumps_activity} propose statistical tests to discriminate between the finite and infinite activity of jumps in a semimartingale. We emphasize that our approach allows one to design a  statistical test in order to identify both the presence and the types of jumps.
\end{itemize}

\section{Examples}\label{section:examples}
In this section, we illustrate the results of Section~\ref{section:main_results} for the class of generalized Laguerre semigroups which have been studied in depth by Patie and Savov in~ \cite{patie2015spectral}. To investigate the behavior of (biorthogonal) spectral projections correlation structure in various scenarios, we first discuss two important examples of subordinators and their inverses.
\begin{example}\label{ex:alpha_st}
Let $\Tau$ be an $\alpha$-stable subordinator, i.e.~in \eqref{eq:varphi}, $\varphi(\lambda)=\lambda^{\alpha}$, $0< \alpha <1$. We recall from \cite{leonenko2013correlation} that for any $\lambda > 0$ and $t \geq 0$, the Laplace transform of its inverse is given by
$$\eta_t(\lambda) = E_{\alpha}(-\lambda t^{\alpha}),$$
where $E_{\alpha}$ is the Mittag-Lefler function defined by $E_{\alpha}(z) = \sum_{k=0}^{\infty}\frac{z^k}{\Gamma(\alpha k + 1)}$ for $z \in \mathbb{C}$.
On the other hand, since $U(ds)=\frac{r^{\alpha-1}}{\Gamma(\alpha)}ds$, $s>0$, we have that
$\E[L_s] = U(0,s) = \frac{s^{\alpha}}{\Gamma(1+\alpha)}$.
Now, Corollary \ref{cor:corr_sub} yields that for any $m,n \in \NN$ and $t \geq s>0$,
\begin{eqnarray*}
\rho_{\nu} \left(\Poly_m \left(X_{\Tau_t}\right), \Nu_n \left(X_{\Tau_s}\right) \right) &=& e^{-\lambda_m^{\alpha}(t-s)} \kappa_{\nu}^{-1}(m)\delta_{mn},\\
\rho_{\nu} \left(\Poly_m \left(X_{\Tau_t}\right), \Poly_n \left(X_{\Tau_s}\right) \right) &=& e^{-\lambda_m^{\alpha}(t-s)}c_{\nu}(n,m).
\end{eqnarray*}
Next, it follows from Theorem~\ref{thm:t-ch} that
\begin{eqnarray*}
\rho_{\nu} \left(\Poly_m \left(X_{L_t}\right), \Poly_n \left(X_{L_s}\right) \right) &=& c_{\nu}(n,m) \left(\lambda_m \int_{0}^s \eta_{t-r}(\lambda_m)U(dr)+\eta_t(\lambda_m) \right)\\
&=& c_{\nu}(n,m) \left(\lambda_m \int_{0}^s E_{\alpha}(-\lambda_m (t-r)^{\alpha})\frac{r^{\alpha-1}}{\Gamma(\alpha)}dr + E_{\alpha}(-\lambda_m t^{\alpha}) \right)\\
&=& \frac{c_{\nu}(n,m)\ \lambda_m  t^\alpha}{\Gamma(\alpha)}\int_0^{s/t}{\frac{E_\alpha(-\lambda_m t^\alpha(1-z)^\alpha)}{z^{1-\alpha}}dz} + c_{\nu}(n,m) E_\alpha(-\lambda_m t^\alpha).
\end{eqnarray*}
Note that since $\varphi$ is strongly regularly varying at $0$, we have, by Proposition~\ref{prop:long_tailed}, that $t \mapsto \eta_t$ is a long-tailed function. Furthermore, it is well known, see e.g.~\cite{leonenko2013correlation}, that when $t \rightarrow \infty$,
\begin{equation*}
\eta_t(\lambda) = E_{\alpha}(-\lambda t^{\alpha}) \sim \frac{1}{\Gamma(1-\alpha)\lambda t^{\alpha}}.
\end{equation*}
Hence, from Theorem~\ref{thm:t-ch}, we deduce that for a fixed $s>0$, when $t \rightarrow \infty$, we have
\begin{eqnarray}
\rho_{\nu}(\Poly_m(X_{L_t}), \Poly_n(X_{L_s})) &\sim &  
  \frac{c_{\nu}(n,m) }{\Gamma(1-\alpha)t^{\alpha}} \left( \frac{1}{\lambda_m} + \frac{s^{\alpha}}{\Gamma(1+\alpha)} \right), \label{eq:alpha}
\end{eqnarray}
and, similarly,
\begin{equation*}
\rho_{\nu}(\Poly_m(X_{L_t}), \Nu_n(X_{L_s})) \sim \frac{\kappa_{\nu}^{-1}(m) \delta_{mn} }{\Gamma(1-\alpha)t^{\alpha}} \left( \frac{1}{\lambda_m} + \frac{s^{\alpha}}{\Gamma(1+\alpha)} \right).
\end{equation*}
When $X$ is a Pearson diffusion, we note that \eqref{eq:alpha} boils down to the case discussed in \cite{leonenko2013correlation}. Finally, since the correlation functions decay in a polynomial rate of $\alpha \in (0,1)$, here the process $X_L$ exhibits long-range dependence.
\end{example}

\begin{example}\label{ex:poisson}
Let $\Tau$ be a Poisson subordinator with mean $\frac{1}{\theta}$, i.e.
in \eqref{eq:varphi}, $\varphi (\lambda) = \theta (1 - e^{-\lambda})$. Then, for the inverse Poisson subordinator, we have that $L_t$ follows $\text{Gam} \left([t+1],1/\theta \right)$, see Leonenko et~al. \cite{leonenko2014correlation}.
Using the moment generating function of a Gamma random variable, we get
\begin{equation}
\eta_t(\lambda)= \int_{0}^{\infty}  {e^{-\lambda s}l_t(ds)} = \left( 1 + \frac{\lambda}{\theta} \right)^{-[t+1]},
\end{equation}
and thus
$U(0,s) = \E[L_s] = \frac{[s+1]}{\theta}.$
Then, it follows from Corollary \ref{cor:corr_sub} that for any $m,n \in \NN$ and $t \geq s>0$,
\begin{eqnarray*}
\rho_{\nu} \left(\Poly_m \left(X_{\Tau_t}\right), \Nu_n \left(X_{\Tau_s}\right) \right) &=& e^{-\theta (1 - e^{-\lambda_m}) (t-s)} \kappa_{\nu}^{-1}(m)\delta_{mn},\\
\rho_{\nu} \left(\Poly_m \left(X_{\Tau_t}\right), \Poly_n \left(X_{\Tau_s}\right) \right) &=& e^{-\theta (1 - e^{-\lambda_m})(t-s)}c_{\nu}(n,m).
\end{eqnarray*}
Now, when $t \rightarrow \infty$, note that
$\eta_t(\lambda) \sim \left( 1 + \frac{\lambda}{\theta} \right)^{-t}.$
Consequently $\lim_{t \rightarrow \infty}\frac{\eta_{t-s}(\lambda_m)}{\eta_t(\lambda_m)} \neq 1$, and therefore \eqref{eq:assympt} and \eqref{eq:assympt_bi} do not hold. However, we are able to compute the exact formulas for the (biorthogonal) spectral projections correlation functions of $X_{L}$ as follows. First, noting that for any $k \in \N$ such that $k < t$, $-[t-k+1]=k-[t+1]$, we have
\begin{eqnarray}
\lambda_m \int_{0}^s \eta_{t-r}(\lambda_m)U(dr)+\eta_t(\lambda_m) &=& \frac{\lambda_m}{\theta} \sum_{k = 0}^{[s]} \left(1 + \frac{\lambda_m}{\theta} \right)^{-[t-k+1]} +  \left(1 + \frac{\lambda_m}{\theta} \right)^{-[t+1]} \nonumber \\
&=& \left(1 + \frac{\lambda_m}{\theta} \right)^{-[t+1]} \left(2 - \left(1 + \frac{\lambda_m}{\theta} \right)^{[s+1]} \right). \label{eq:poisson}
\end{eqnarray}
Thus, it follows from Theorem \ref{thm:t-ch} that
\begin{eqnarray*}
\rho_{\nu}(\Poly_m(X_{L_t}), \Poly_n(X_{L_s})) &=& c_{\nu}(n,m) \left(1 + \frac{\lambda_m}{\theta} \right)^{-[t+1]} \left(2 - \left(1 + \frac{\lambda_m}{\theta} \right)^{[s+1]} \right), \\
\rho_{\nu}(\Poly_m(X_{L_t}), \Nu_n(X_{L_s})) &=& \kappa_{\nu}^{-1}(m) \delta_{mn} \left(1 + \frac{\lambda_m}{\theta} \right)^{-[t+1]} \left(2 - \left(1 + \frac{\lambda_m}{\theta} \right)^{[s+1]} \right).
\end{eqnarray*}
Since the spectral projections correlation functions decay in an exponential rate, the time-changed process $X_{L}$ exhibits short-range dependence although it is non-Markovian.
\end{example}

\subsection{A short review of the generalized Laguerre semigroups} \label{section:gL}
In this section, we provide a short description of the so-called generalized Laguerre processes introduced and studied by Patie and Savov~\cite{patie2015spectral}, see also Patie~et~al.~\cite{patie_savov_zhao}. We point out that these processes have been recently used to model asset price dynamics in Jarrow~et~al.~\cite{risk_neutral}.
To this end, let $\tilde{\mathbf{A}}$ be the infinitesimal generator of classical Laguerre process which in financial literature is known as a Cox-Ingersoll-Ross (CIR) process, i.e.~for at least $f \in C^2_0(\R_+)$, we have
\begin{equation} \label{eq:CIR-gen}
\tilde{\mathbf{A}}f(x) = \sigma^2 x f''(x)+(\beta+\sigma^2-x)f'(x),
\end{equation}
where $\beta, \sigma \geq 0$. We say that a semigroup $P = (P_t)_{t \geq 0}$ is a \emph{generalized Laguerre (gL) semigroup} if its infinitesimal generator is given, for a smooth function $f$ on $x >0$, by
\begin{equation}\label{eq:gCIR-gen}
\mathbf{A}f(x) = \tilde{\mathbf{A}}f(x)+\int_0^{\infty}\left( f(e^{-y}x)-f(x)+yxf'(x) \right)\Pi(x,dy),
\end{equation}
where $\Pi(x,dy) = \frac{\Pi(dy)}{x}$, with ${\Pi}$ being a L\'evy measure concentrated on $(0,\infty)$ and satisfying the integrability condition $\int_0^{\infty}(y^2 \wedge y)\Pi(dy)<\infty$. We call the corresponding process $X=(X_t)_{t \geq 0}$ a \emph{generalized Laguerre process}. Note that when $\Pi(0,\infty)=0$, then $P$ boils down to the semigroup of a classical Laguerre process. Moreover, from ~\cite[Theorem 1.6]{patie2015spectral} we have that the semigroup $P$ admits a unique invariant measure, which in this case is absolutely continuous with a density that we denote by $\nu$, and write the Hilbert space $L^2(\nu)$ as in Section \ref{intro}. Recall that $P$ can be extended to a contraction semigroup in $L^2(\nu)$, and by an abuse of notation, we still denote it by $P$. Now,~\cite[Theorem 1.11]{patie2015spectral} yields that if $\overline{\Pi}(y) = \int_y^{\infty}\Pi(dr)$ is strongly regularly varying at $0$ with some index $\alpha \in (0,1)$, then, for any $f \in L^2(\nu)$ and $t > T_{\Pi}$ for some explicit $T_{\Pi}$ (with $T_{\Pi}=0$ when $\sigma^2>0$), we have the following spectral expansion,
\begin{equation*}\label{eq:spect_exp}
P_tf=\sum_{n=0}^{\infty}e^{-nt} \langle f, \mathcal{V}_n \rangle_{\nu}\Poly_n \quad \text{in } L^2(\nu),
\end{equation*}
where $(\Poly_n, \Nu_n)_{n \geq 0}$ form a biorthogonal sequence of $L^2(\nu)$, and are expressed as follows:
\begin{equation*}
\Poly_n(x)= \sum_{k=0}^{n} (-1)^k \frac{\binom{n}{k}}{W_\phi (k+1)}x^k \in L^2(\nu),
\end{equation*}
and,
\begin{equation*}
\mathcal{V}_n(x) = \frac{\mathcal{R}^{(n)}\nu(x)}{\nu(x)} = \frac{(x^n\nu(x))^{(n)}}{n!\nu(x)} \in L^2(\nu),
\end{equation*}
with the last equation serving as a definition of the Rodrigues operator. Here, $W_{\phi}(1)=1$ and, for $n \in \mathbb{N}$, $W_{\phi}(n+1)=\prod_{k=1}^n \phi(k)$, where $\phi$ is the Bernstein function, see~\eqref{eq:varphi1}, which takes the form
\begin{equation*}
\phi(\lambda) = \beta + \sigma^2 \lambda + \int_0^{\infty} \left(1 - e^{-\lambda y} \right) \overline{\Pi}(y)dy,
\end{equation*}
with $\Pi, \beta, \sigma^2$ as in \eqref{eq:gCIR-gen}. Furthermore, by \cite[Theorem 7.3 and Proposition 8.4]{patie2015spectral} we have, that for any $n \geq 0$ and $t \geq 0$, $\Poly_n$ (resp.~$\Nu_n$) is an eigenfunction for $P_t$ (resp. $P_t^*$) associated to the eigenvalue $e^{-nt}$, i.e.~$\Poly_n, \Nu_n \in L^2(\nu)$ and
\begin{equation*}
P_t \Poly_n(x) = e^{-nt}\Poly_n(x) \quad \text{and}\quad P_t^* \Nu_n(x) = e^{-nt}\Nu_n(x),
\end{equation*}
with $(P^*_t)_{t \geq 0}$ being the adjoint of $(P_t)_{t \geq 0}$ in $L^2(\nu)$. Therefore, in this case, we have $\lambda_n = n$, $n \in \N$.\\
\\
Next, we describe the eigenvalue expansions of specific instances of the generalized Laguerre semigroups which illustrate the different situations that are ranging from the self-adjoint case to perturbation of a self-adjoint differential operator through non-local operators without diffusion component. We study their spectral projections correlation structure, and discuss some of their important properties as are range dependence and symmetry (self-adjointness), among others.

\subsection{The self-adjoint diffusion case}\label{section:self_adj} For any $\beta>0$,  the infinitesimal generator of the classical  Laguerre process takes the form
\begin{equation*}
\mathbf{A}_{\beta} f(x) = xf''(x)+(\beta+1-x)f'(x).
\end{equation*}
Note that this is the infinitesimal generator of a one-dimensional diffusion often referred in the literature as the CIR process. The eigenfunctions are given by
\begin{equation*}
\overline{\mathcal{L}}^{(\beta)}_n(x) =\sqrt{\mathfrak{c}_n(\beta)}
\mathcal{L}^{(\beta)}_n(x),
\end{equation*}
where $\mathfrak{c}_n(\beta)= \frac{\Gamma(n+1)\Gamma(\beta+1)}{\Gamma(n+\beta+1)}$ and $\mathcal{L}_n^{(\beta)}(x) = \sum_{k=0}^{n}(-1)^{k} \binom{n+\beta}{n-k}\frac{x^k}{k!}$ is the associated Laguerre polynomial of order $\beta$. Denote by
\begin{equation}\label{eq:gamma_beta}
\gamma_{\beta}(dx)=\frac{x^{\beta} e^{-x}}{\Gamma(\beta+1)}dx, \quad x>0,
\end{equation}
the law of a Gamma random variable with parameter $(\beta+1)$. Then, the semigroup  is self-adjoint in $L^2(\gamma_{\beta})$ and the sequence $(\overline{\mathcal{L}}^{(\beta)}_n)_{n \geq 0}$ forms an orthonormal basis of $L^2(\gamma_{\beta})$.
In particular, this means that, for $n,m \in \N$, we have that
\begin{equation*}
\kappa_{\gamma_{\beta}}(m) = 1 \quad \text{and}\quad c_{\gamma_{\beta}}(n,m) = \delta_{nm}.
\end{equation*}
Hence, it follows from Theorem ~\ref{thm:corr}, that for $m,n\in \NN$ and $t \geq s > 0$,
\begin{equation*}
\rho_{\gamma_{\beta}}(\overline{\mathcal{L}}^{(\beta)}_m(X_t), \overline{\mathcal{L}}^{(\beta)}_n(X_s)) = e^{-m(t-s)}\delta_{nm}.
\end{equation*}
Now, from Corollary \ref{cor:corr_sub}, Theorem \ref{thm:t-ch} and examples \ref{ex:alpha_st} and \ref{ex:poisson} we have the following additional results.
\begin{itemize}
\item Let $\Tau$ be an \textbf{$\alpha$-stable subordinator}, i.e.~$\varphi(\lambda)=\lambda^{\alpha}$, $0<\alpha<1$, see Example~\ref{ex:alpha_st}. Then, for any $t \geq s > 0$,
\begin{equation*}
 \rho_{\gamma_{\beta}}(\overline{\mathcal{L}}^{(\beta)}_m(X_{\Tau_t}), \overline{\mathcal{L}}^{(\beta)}_n(X_{\Tau_s})) = e^{-m^{\alpha}(t-s)} \delta_{nm}.
\end{equation*}
\item Let $\Tau$ be a \textbf{Poisson subordinator} with parameter $\theta$, i.e.~$\varphi(\lambda)=\theta(1-e^{-\lambda})$, see Example~\ref{ex:poisson}. Then, for any $t \geq s > 0$,
 \begin{equation*}
 \rho_{\gamma_{\beta}}(\overline{\mathcal{L}}^{(\beta)}_m(X_{\Tau_t}), \overline{\mathcal{L}}^{(\beta)}_n(X_{\Tau_s})) = e^{-\theta(1-e^{-m})(t-s)} \delta_{nm}.
\end{equation*}
\item Let $L$ be the \textbf{inverse of an $\alpha$-stable subordinator}, see Example~\ref{ex:alpha_st}. Then,
\begin{equation*}
\rho_{\gamma_{\beta}} \left(\overline{\mathcal{L}}^{(\beta)}_m \left(X_{L_t}\right) ,\overline{\mathcal{L}}^{(\beta)}_n \left(X_{L_s}\right) \right) =\frac{\delta_{nm}\ m  t^\alpha}{\Gamma(\alpha)}\int_0^{s/t}{\frac{E_\alpha(-m t^\alpha(1-z)^\alpha)}{z^{1-\alpha}}dz} +\delta_{nm} E_\alpha(-m t^\alpha).
\end{equation*}
Furthermore, for a fixed $s>0$, when $t \rightarrow \infty$,
\begin{equation*}
\rho_{\gamma_{\beta}}(\overline{\mathcal{L}}^{(\beta)}_m(X_{L_t}), \overline{\mathcal{L}}^{(\beta)}_n(X_{L_s})) \sim  \frac{\delta_{nm} }{\Gamma(1-\alpha)t^{\alpha}} \left( \frac{1}{m} + \frac{s^{\alpha}}{\Gamma(1+\alpha)} \right).
\end{equation*}
\item Let $L$ be the \textbf{inverse of a Poisson subordinator} with parameter $\theta$, see Example~\eqref{ex:poisson}. Then, for any $t \geq s > 0$,
\begin{equation*}
\rho_{\gamma_{\beta}}(\overline{\mathcal{L}}^{(\beta)}_m(X_{L_t}), \overline{\mathcal{L}}^{(\beta)}_n(X_{L_s})) = \delta_{nm} \left(1 + \frac{m}{\theta} \right)^{-[t+1]} \left(2 - \left(1 + \frac{m}{\theta} \right)^{[s+1]} \right).
\end{equation*}
\end{itemize}

\subsection{Small perturbation of the Laguerre semigroup.}\label{section:small_pert}
Let $b \geq 1$, and take $\sigma^2=1$, $\beta = \frac{b^2-1}{b}$ and $\overline{\Pi}(y)=e^{-b y}$, $y \geq 0$ in \eqref{eq:gCIR-gen}, i.e.~we consider, for $f$ smooth,
\begin{equation*}
\mathbf{A}^{(b)}f(x)=xf''(x)+ \left( \frac{{b}^2-1}{b} +1-x \right)f'(x)+\frac{b}{x}\int_0^{\infty}(f(e^{-y}x)-f(x)+yxf'(x))e^{-b y}dy.
\end{equation*}
The associated semigroup  is ergodic with a unique invariant measure $\nu_b$,
$$\nu_b(dx)=\frac{(1+x)}{b+1}\gamma_{b-1}(dx), \quad x>0,$$
with $\gamma_{b-1}(dx)$ as in ~\eqref{eq:gamma_beta}. Then, the eigenfunctions and co-eigenfunctions $(\Poly^{(b)}_n, \mathcal{V}^{(b)}_n)_{n \geq 0}$ are expressed in terms of Laguerre polynomials $\left( \mathcal{L}_n^{(b)} \right)_{n \geq 0}$ as follows, $n \geq 0$,
\begin{eqnarray*}
\Poly^{(b)}_n(x) &=& \mathfrak{c}_n(b+1)\mathcal{L}_{n}^{(b+1)}(x)-\frac{\mathfrak{c}_n(b+1)}{b}x\mathcal{L}_{n-1}^{(b+2)}(x) , \label{eq:eigen_perturb}\\
\mathcal{V}^{(b)}_n(x)&=&\frac{1}{x+1}\mathcal{L}_{n}^{(b-1)}(x)+\frac{x}{x+1}\mathcal{L}_{n}^{(b)}(x),\label{eq:coeigen_perturb}
\end{eqnarray*}
where $\mathfrak{c_n(\cdot)}$ and $\left( \mathcal{L}_n^{(b)} \right)_{n \geq 0}$ are defined as in Section~\ref{section:self_adj}, see \cite[Example 3.2]{patie2015spectral}. Next, it follows from \cite[Theorem 2.2 and (3.9)]{patie2015spectral} that
\begin{equation*}
\Vert \Poly^{(b)}_n \Vert_{\nu_b} = O(1), \quad \Vert \Nu^{(b)}_n \Vert_{\nu_b} = O(n^{(b + 1)/2}).
\end{equation*}
Then, for any $m \in \N$, we have
\begin{eqnarray*}
\kappa_{\nu_b}(m) &=& O(m^{(b + 1)/2}),
\end{eqnarray*}
and therefore, the biorthogonal sequence $(\Poly^{(b)}_n, \Nu^{(b)}_n)_{n \in \N}$ is tame, see Remark~\ref{rem:tame} for definition. Thus, it follows from Theorem \ref{thm:corr} that
\begin{equation*}
\rho_{\nu_b}(\Poly^{(b)}_m(X_t), \Nu^{(b)}_n(X_s)) = e^{-m(t-s)} \kappa_{\nu_b}^{-1}(m)  \delta_{nm}.
\end{equation*}
Now, from Corollary ~\ref{cor:corr_sub}, Theorem ~\ref{thm:t-ch} and examples ~\ref{ex:alpha_st} and ~\ref{ex:poisson} we have the following results.
\begin{itemize}
\item Let $\Tau$ be an \textbf{$\alpha$-stable subordinator}, i.e.~$\varphi(\lambda)=\lambda^{\alpha}$, $0<\alpha<1$, see Example~\ref{ex:alpha_st}. Then, for any $t \geq s>0$,
\begin{eqnarray*}
 \rho_{\nu_b}(\Poly^{(b)}_m(X_{\Tau_t}), \Nu^{(b)}_n(X_{\Tau_s}))= e^{-m^{\alpha}(t-s)}\kappa_{\nu_b}^{-1}(m) \delta_{nm}.
\end{eqnarray*}
\item Let $\Tau$ be a \textbf{Poisson subordinator} with parameter $\theta$, i.e.~$\varphi(\lambda)=\theta(1-e^{-\lambda})$, see Example~\ref{ex:poisson}. Then, for any $t \geq s > 0$
\begin{eqnarray*}
\rho_{\nu_b}(\Poly^{(b)}_m(X_{\Tau_t}), \Nu^{(b)}_n(X_{\Tau_s})) = e^{-\theta(1-e^{-m})(t-s)} \kappa_{\nu_b}^{-1}(m) \delta_{nm}.
\end{eqnarray*}
\item Let $L$ be the \textbf{inverse of an $\alpha$-stable subordinator}, see Example~\ref{ex:alpha_st}. Then, for a fixed $s>0$, when $t \rightarrow \infty$,
\begin{eqnarray*}
\rho_{\nu_b}(\Poly^{(b)}_m(X_{L_t}), \Nu^{(b)}_n(X_{L_s})) \sim \frac{\kappa_{\nu_b}^{-1}(m)\delta_{mn} }{\Gamma(1-\alpha)t^{\alpha}} \left( \frac{1}{m} + \frac{s^{\alpha}}{\Gamma(1+\alpha)} \right).
\end{eqnarray*}
\item Let $L$ be the \textbf{inverse of a Poisson subordinator} with parameter $\theta$, see Example~\ref{ex:poisson}. Then, for any $t \geq s > 0$,
\begin{eqnarray*}
\rho_{\nu_b}(\Poly^{(b)}_m(X_{L_t}), \Nu^{(b)}_n(X_{L_s})) = \kappa_{\nu_b}^{-1}(m)\delta_{mn}\left(1 + \frac{m}{\theta} \right)^{-[t+1]} \left(2 - \left(1 + \frac{m}{\theta} \right)^{[s+1]} \right).
\end{eqnarray*}
\end{itemize}

\subsection{The Gauss-Laguerre semigroup.}\label{section:big_jumps}
We next consider the Gauss-Laguerre semigroup $P^{\alpha, b} = (P^{\alpha, b}_t)_{t \geq 0}$ which has been introduced and extensively studied in~\cite{patie_gauss_laguerre}, and which is an instance of the generalized Laguerre semigroups, see \cite[Example 3.3]{patie2015spectral}. In particular, its infinitesimal generator, for any $\alpha \in (0,1)$ and $b \in [1-\frac{1}{\alpha},\infty]$, and for any given smooth function $f$, takes the form
\begin{equation*}
\mathbf{A}^{(\alpha,b)}f(x) = (b_{\alpha}-x)f'(x)+\frac{\sin(\alpha \pi)}{\pi}x\int_0^1 f''(xy)g_{\alpha, b}(y)dy, \ x>0,
\end{equation*}
where $b_{\alpha} = \frac{\Gamma(\alpha b + \alpha + 1)}{\Gamma(\alpha b + 1)}$ and
\[ g_{\alpha, b}(y)=\frac{\Gamma(\alpha)}{b + \frac{1}{\alpha}+1}y^{b + \frac{1}{\alpha}+1}{}_2F_1(\alpha(b+1)+1,\alpha+1;\alpha(b+1)+2;y^{\frac{1}{\alpha}}), \]
with ${}_2F_1$ the Gauss hypergeometric function. The associated semigroup $P^{\alpha, b} = (P^{\alpha, b}_t)_{t \geq 0}$ is a non-self-adjoint contraction in $L^2(\mathbf{e}_{\alpha,b})$, where
\[ \mathbf{e}_{\alpha,b}(dx) = \frac{x^{b+\frac{1}{\alpha}-1}e^{-x^{\frac{1}{\alpha}}}}{\Gamma(\alpha b + 1)}dx, \ x>0, \]
is its unique invariant measure. For any $x \geq 0$, we set $\mathcal{P}^{(\alpha,b)}_0(x) = 1$ and for any $n \geq 1$,
\begin{eqnarray*}
\mathcal{P}^{(\alpha,b)}_n(x) &=& \Gamma(\alpha b +1)\sum_{k=0}^{n}(-1)^k \frac{\binom{n}{k}}{\Gamma(\alpha k + \alpha b + 1)}x^k, \\
\Nu^{(\alpha,b)}_n(x) &=& \frac{(-1)^n}{n!\ \mathbf{e}_{\alpha, b}(x)}(x^n \mathbf{e}_{\alpha, b}(x))^{(n)},
\end{eqnarray*}
which are the eigenfunctions and co-eigenfunctions of $P^{\alpha, b}$. It is worth mentioning that in~\cite[Proposition 3.3]{patie_gauss_laguerre} the authors show that the $\Nu^{(\alpha,b)}_n$'s can be expressed in terms of sequences of polynomials as well. Then, it follows from \cite[Proposition 2.3]{patie_gauss_laguerre} and \cite[Theorem 2.2]{patie2015spectral} that
\begin{equation*}
\Vert \Poly^{(\alpha,b)}_n \Vert_{\mathbf{e}_{\alpha, b}} = O(1),\quad \Vert \Nu^{(\alpha,b)}_n \Vert_{\mathbf{e}_{\alpha, b}} = O\left(e^{T_{\alpha}n}\right),
\end{equation*}
where $T_{\alpha}=-\ln (2^{\alpha}-1)$. Then, we have that for any $m \in \N$,
\begin{eqnarray*}
\kappa_{\mathbf{e}_{\alpha,b}}(m) = O\left(e^{T_{\alpha}m}\right),
\end{eqnarray*}
hence, in this case. the biorthogonal sequence $(\Poly^{(\alpha,b)}_n, \Nu^{(\alpha,b)}_n)_{n \in \N}$ is a  wild system in $L^2(\mathbf{e}_{\alpha, b})$. Now, it follows from Theorem ~\ref{thm:corr} that
\begin{eqnarray*}
\rho_{\mathbf{e}_{\alpha,b}}(\Nu^{(\alpha,b)}_n(X_s), \Poly^{(\alpha,b)}_m(X_t)) = e^{-m(t-s)}  \kappa_{\mathbf{e}_{\alpha,b}}^{-1}(m) \delta_{nm}.
\end{eqnarray*}
Next, from Corollary ~\ref{cor:corr_sub}, Theorem ~\ref{thm:t-ch} and examples ~\ref{ex:alpha_st} and ~\ref{ex:poisson} we have the following additional results.
\begin{itemize}
\item Let $\Tau$ be an \textbf{$\alpha$-stable subordinator}, i.e.~$\varphi(\lambda)=\lambda^{\alpha}$, $0<\alpha<1$, see Example~\ref{ex:alpha_st}. Then, for any $t \geq s>0$,
\begin{eqnarray*}
 \rho_{\mathbf{e}_{\alpha,b}}(\Poly^{(\alpha,b)}_m(X_{\Tau_t}), \Nu^{(\alpha,b)}_n(X_{\Tau_s})) = e^{-m^{\alpha}(t-s)} \kappa_{\mathbf{e}_{\alpha,b}}^{-1}(m)  \delta_{nm}.
\end{eqnarray*}
\item Let $\Tau$ be a \textbf{Poisson subordinator} with parameter $\theta$, i.e.~$\varphi(\lambda)=\theta(1-e^{-\lambda})$, see Example~\ref{ex:poisson}. Then, for any $ t \geq s > 0$
 \begin{eqnarray*}
 \rho_{\mathbf{e}_{\alpha,b}}(\Poly^{(\alpha,b)}_m(X_{\Tau_t}), \Nu^{(\alpha,b)}_n(X_{\Tau_s})) = e^{-\theta(1-e^{-m})(t-s)}  \kappa_{\mathbf{e}_{\alpha,b}}^{-1}(m) \delta_{nm}.
\end{eqnarray*}
\item Let $L$ be the \textbf{inverse of an $\alpha$-stable subordinator}, see Example~\ref{ex:alpha_st}. Then, for a fixed $s>0$, when $t \rightarrow \infty$,
\begin{eqnarray*}
\rho_{\mathbf{e}_{\alpha,b}}(\Poly^{(\alpha,b)}_m(X_{L_t}), \Nu^{(\alpha,b)}_n(X_{L_s})) \sim  \frac{\kappa_{\mathbf{e}_{\alpha,b}}^{-1}(m)\delta_{mn}}{\Gamma(1-\alpha)t^{\alpha}} \left( \frac{1}{m} + \frac{s^{\alpha}}{\Gamma(1+\alpha)} \right).
\end{eqnarray*}
\item Let $L$ be the \textbf{inverse of a Poisson subordinator} with parameter $\theta$, see Example~\ref{ex:poisson}. Then, for any $t \geq s > 0$,
\begin{eqnarray*}
\rho_{\mathbf{e}_{\alpha,b}}(\Poly^{(\alpha,b)}_m(X_{L_t}), \Nu^{(\alpha,b)}_n(X_{L_s})) = \kappa_{\mathbf{e}_{\alpha,b}}^{-1}(m)\delta_{mn}\left(1 + \frac{m}{\theta} \right)^{-[t+1]} \left(2 - \left(1 + \frac{m}{\theta} \right)^{[s+1]} \right).
\end{eqnarray*}
\end{itemize}

\section{Proofs of the main results}\label{section:proofs}
\subsection{Proof of Theorem \ref{thm:corr}}\label{sec:proof_corr}
We split the proof of Theorem ~\ref{thm:corr} into several intermediary lemmas.
\begin{lemma}\label{lemma:biorth}
The sequence $(\Poly_n, \Nu_n)_{n \in \mathtt{N}}$, defined in Assumption \ref{assmp1}, form a biorthogonal sequence in $L^2(\nu)$, i.e.~ for any $n, m \in \mathtt{N}$,
\begin{equation}
\langle \Poly_n, \Nu_m \rangle_{\nu} = \delta_{nm}.
\end{equation}
\end{lemma}
\begin{proof}
First, recall that in Section \ref{sec:prelim} we assumed, without loss of generality, that for any $n \in \NN$, $\langle \Poly_n, \Nu_n \rangle_{\nu} =1$. Therefore, we need to show that $\langle \Poly_n, \Nu_m \rangle_{\nu} =0$ when $n \neq m$. Then, note that for all $t \geq 0$ and $m, n \in \NN$,
\begin{equation*}
\langle \Poly_n, \Nu_m \rangle_{\nu} = e^{\lambda_n t} \langle P_t \Poly_n, \Nu_m \rangle_{\nu} =  e^{\lambda_n t} \langle \Poly_n, P_t^* \Nu_m \rangle_{\nu} = e^{(\lambda_n -\lambda_m) t} \langle \Poly_n,  \Nu_m \rangle_{\nu}
\end{equation*}
where in the first and last equality we used \eqref{eq:eigen} and \eqref{eq:coeigen} respectively. Therefore,
\begin{equation*}
\left( 1 - e^{(\lambda_n -\lambda_m) t} \right) \langle \Poly_n, \Nu_m \rangle_{\nu} = 0.
\end{equation*}
Hence, since we assumed that the eigenvalues are of multiplicity $1$, $\lambda_n \neq \lambda_m$ if $n \neq m$. Thus, $\langle \Poly_n, \Nu_m \rangle_{\nu} = 0$, which concludes the proof.
\end{proof}

\begin{lemma}\label{lemma:std}
Let $f \in L^2(\nu)$. Then, for any $t \geq 0$,
\begin{equation}
std_{\nu}(f(X_t)) = \sqrt{\nu f^2 - (\nu f)^2}=\sqrt{\nu f^2}.
\end{equation}
In particular, if $f$ is such that $\nu f = 0$, then
\begin{equation}
std_{\nu}(f(X_t)) = \Vert f \Vert_{\nu}.
\end{equation}
\end{lemma}

\begin{proof}
The first claim immediately follows, for any $t \geq 0$, from the sequence of equalities
\begin{equation*}
std_{\nu}(f(X_t)) = \sqrt{\nu P_t f^2 - (\nu P_t f)^2} =  \sqrt{\nu f^2 - (\nu f)^2}.
\end{equation*}
Finally, if $\nu f = 0$, then we have
\begin{equation}\label{eq:std}
std_{\nu}(f(X_t))=\sqrt{\nu f^2} =\Vert f \Vert_{\nu}.
\end{equation}
\end{proof}

\begin{lemma}\label{lemma:cov}
Let $f \in L^2(\nu)$. Then, for any $m \in \NN$ and $t \geq s > 0$,
\begin{equation*}
\C_{\nu}(\Poly_m(X_t), f(X_s)) = e^{-\lambda_m (t-s)} \langle \Poly_m, f \rangle_{\nu}.
\end{equation*}
\end{lemma}

\begin{proof}
First,  Lemma~\ref{lemma:biorth} yields that $\langle \Poly_m,\Nu_n \rangle_\nu = \delta_{mn}$, $m,n \in \NN$. In particular, since $\nu$ is invariant, for any $m \in \NN$, $\nu P_t \Poly_m = \nu \Poly_m = \langle \Poly_m, 1 \rangle_\nu =\delta_{0m}=0$, where we used the fact that the constant function $\mathbbm{1}$ is an eigenfunction for $P_t$ since $P_t \mathbbm{1} = \mathbbm{1}$ for all $t \geq 0$. Similarly, since $P_0 = P_0^* = \mathbbm{1}$, then $\nu \Poly_m = \nu \Nu_m =\delta_{0m}= 0$. Then, from the definition of the covariance function given in \eqref{eq:def_cov},  we obtain, for any $t \geq s > 0$,
\begin{eqnarray*}
\C_{\nu}(\Poly_m(X_t), f(X_s)) &=& \E_{\nu} [\Poly_m(X_t)f(X_s) ] - \E_{\nu} [\Poly_m(X_t)]\E_{\nu} [f(X_s)]\\
&=& \E_{\nu} [ \Poly_m(X_t)f(X_s)] - \nu P_t \Poly_m \ \nu P_s f  \\
&=&  \E_{\nu} [\Poly_m(X_t)f(X_s) ].
\end{eqnarray*}
Next, using the Markov property and \eqref{eq:eigen}, we get
\begin{eqnarray*}
\E_{\nu}[\Poly_m(X_{t})f(X_s)]
&=& \E_{\nu}[\E_{X_s}[\Poly_m(X_{t-s})]f(X_s)]  \\
&=& \E_{\nu}[P_{t-s} \Poly_m(X_s)f(X_s)]  \\
&=& e^{-\lambda_m (t-s)} \E_{\nu}[\Poly_m(X_{s})f(X_s)]  \\
&=& e^{-\lambda_m (t-s)} \nu P_s \Poly_m f = e^{-\lambda_m (t-s)} \nu \Poly_m f\\
&=& e^{-\lambda_m (t-s)} \langle \Poly_m, f \rangle_\nu,
\end{eqnarray*}
where in the second last equality we used the fact that $\nu$ is an invariant measure for $P$.
\end{proof}

We are now ready to prove Theorem \ref{thm:corr}. First, recall from the proof of Lemma~\ref{lemma:cov} that for any $m\in \NN$, $\nu \Poly_m = \nu \Nu_m = 0$.   Next, it follows from Lemma \ref{lemma:std} that for any $t \geq 0$ and $m \in \NN$,
\begin{equation*}
std_{\nu}(\Poly_m(X_t)) = \Vert \Poly_m \Vert_{\nu} \quad \text{and}\quad std_{\nu}(\Nu_m(X_t)) = \Vert \Nu_m \Vert_{\nu}.
\end{equation*}
Then, using Lemma \ref{lemma:cov} with $f = \Nu_n$ and $f = \Poly_m$, respectively, we get, for any $t \geq s > 0$ and $n,m \in \NN$, that
\begin{eqnarray*}
\rho_{\nu}(\Poly_m(X_t), \Poly_n(X_s)) &=& \frac{e^{-\lambda_m(t-s)} \langle \Poly_n, \Poly_m \rangle_{\nu}}{\Vert \Poly_n \Vert_{\nu} \Vert \Poly_m \Vert_{\nu}} = e^{-\lambda_m(t-s)} c_{\nu}(n,m),\\
\rho_{\nu}(\Poly_m(X_t), \Nu_n(X_s)) &=& \frac{e^{-\lambda_m (t-s)} \delta_{mn}}{\Vert \Poly_m \Vert_{\nu} \ \Vert \Nu_m \Vert_{\nu}} = e^{-\lambda_m (t-s)} \kappa_{\nu}^{-1}(m) \delta_{mn},
\end{eqnarray*}
where we recall that for $m, n \in \NN$, $\kappa_{\nu}(m) = \Vert \Poly_m \Vert_{\nu} \Vert \Nu_m \Vert_{\nu}$ and $c_{\nu}(n,m) = \frac{\langle \Poly_n, \Poly_m \rangle_{\nu}}{\Vert \Poly_n \Vert_{\nu} \Vert \Poly_m \Vert_{\nu}}$. Then, by symmetry, it is easy to note that for any $t,s > 0$, we have
\begin{equation*}
\rho_{\nu}(\Poly_m(X_t), \Poly_n(X_s)) = e^{-\lambda_m(t-s)^+-\lambda_n(s-t)^+} c_{\nu}(n,m).
\end{equation*}
Finally, the Cauchy-Schwartz inequality entails that $|\langle \Poly_n, \Poly_m \rangle_{\nu}| \le \Vert \Poly_n \Vert_{\nu} \Vert \Poly_m \Vert_{\nu}$ and hence  $-1 \le c_{\nu}(n,m) \le 1$. Moreover, when $n = m$, we have that for any $n \in \NN$,
\begin{equation*}
c_{\nu}(n,n) =  \frac{\langle \Poly_n, \Poly_n \rangle_{\nu}}{\Vert \Poly_n \Vert_{\nu} \Vert \Poly_n \Vert_{\nu}} = \frac{\Vert \Poly_n \Vert_{\nu}^2}{\Vert \Poly_n \Vert_{\nu}^2}=1,
\end{equation*}
and this concludes the proof of Theorem \ref{thm:corr}. $\blacksquare$

\subsection{Proof of Lemma \ref{lemma:corr}}\label{sec:proof_lemma_corr}
The definitions of the covariance and correlation functions in \eqref{eq:def_cov} and \eqref{eq:def_corr} give that for any $t \geq 0$,
\begin{eqnarray}
\rho_{\nu}(f(X_t),g(X_t)) &=& \frac{\E_{\nu}[f(X_t)g(X_t)] - \E_{\nu}[f(X_t)]\E_{\nu}[g(X_t)]}{std_{\nu}(f(X_t))std_{\nu}(g(X_t))} \nonumber \\
&=& \frac{\nu P_t fg - \nu P_t f \cdot \nu P_t g}{\sqrt{\nu f^2 - (\nu f)^2} \cdot \sqrt{\nu g^2 - (\nu g)^2}}\nonumber \\
&=& \frac{\nu fg - \nu f \cdot \nu g}{\sqrt{\nu f^2 - (\nu f)^2} \cdot \sqrt{\nu g^2 - (\nu g)^2}}\nonumber \\
&=& \frac{\langle f,g \rangle_{\nu} - \nu f \cdot \nu g}{\sqrt{\nu f^2 - (\nu f)^2} \cdot \sqrt{\nu g^2 - (\nu g)^2}} \label{eq:corr_nu}
\end{eqnarray}
where in the third equality we used the fact that $\nu$ is an invariant measure for $P_t$. Next, for $n,m \in \NN$, taking $f = \Poly_m$ with $g = \Poly_n$ and $g = \Nu_n$ in \eqref{eq:corr_nu}, and using Lemma \ref{lemma:biorth} and Lemma \ref{lemma:std}, we get, for any $t \geq 0$,
\begin{eqnarray*}
\rho_{\nu}(\Poly_m(X_t),\Poly_n(X_t)) &=& \frac{\langle \Poly_m,\Poly_n \rangle_{\nu}}{\Vert \Poly_m \Vert_{\nu} \cdot \Vert \Poly_n \Vert_{\nu}} = c_{\nu}(n,m),\\
\rho_{\nu}(\Poly_m(X_t),\Nu_n(X_t)) &=& \frac{\langle \Poly_m,\Nu_n \rangle_{\nu}}{\Vert \Poly_m \Vert_{\nu} \cdot \Vert \Nu_n \Vert_{\nu}} = \kappa_{\nu}^{-1}(m)\delta_{mn}
\end{eqnarray*}
where we recall that for any $n \in \NN$, $\nu \Poly_n = \nu \Nu_n =0$. $\blacksquare$

\subsection{Proof of Proposition \ref{prop:long_tailed}}\label{sec:proof_long_tailed}
For a function $f$, we write $\mathcal{L}_f(q)=\int_0^{\infty} e^{-qz}f(z)dz$ and we use the same notation for the Laplace transform of a measure. Then, for any $\lambda>0$, denoting
\begin{equation}\label{eq:gamma_m}
U_{\lambda}(dw) = \int_{0}^{\infty}\Prob(\Tau_z \in dw)e^{-\lambda z}dz, \quad w \geq 0,
\end{equation}
the ${\lambda}$-potential measure of $\Tau$, we have, for any $q>0$,
\begin{eqnarray}
\mathcal{L}_{U_{\lambda}}(q) &=& \int_{0}^{\infty}e^{-qw}\int_{0}^{\infty}\Prob(\Tau_z \in dw)e^{-\lambda z}dz \nonumber \\
&=&\int_{0}^{\infty} e^{-\lambda z}  \int_{0}^{\infty}e^{-qw}\Prob(\Tau_z \in dw) \nonumber \\
&=& \int_{0}^{\infty} e^{-\lambda z} e^{-z \varphi(q)}dz = \frac{1}{\lambda +\varphi(q)}. \label{eq:gamma_laplace}
\end{eqnarray}
Next, as $\varphi(0)=0$, see~\eqref{eq:varphi1}, $\int_0^{\infty}U_{\lambda}(dw)=\frac{1}{\lambda}$. Thus, writing, for any $t\geq0$,  $\overline{U}_{\lambda}(t) = \lambda\int_{t}^{\infty}U_{\lambda}(dw)$ and changing the order of integration justified by an application of Tonelli's theorem, we get that for any $q>0$,
\begin{equation}
\mathcal{L}_{\overline{U}_{\lambda}}(q) = \frac{1}{q} - \lambda \frac{1}{q(\lambda+\varphi(q))} = \frac{\varphi(q)}{q(\lambda+\varphi(q))}.
\end{equation}
On the other hand, it is well known that the Laplace transform of $t\mapsto \eta_{t}(\lambda)$, where we recall that  $\eta_{t}(\lambda)= \int_0^{\infty}e^{-\lambda s}l_t(ds)$, takes the form, for any $q>0$,
\begin{equation}\label{eq:double_Laplace}
\mathcal{L}_{\eta_{\cdot}(\lambda)}(q) = \frac{\varphi(q)}{q(\lambda +\varphi(q))},
\end{equation}
see e.g.~Mijena and Nane \cite{mijena2014corr}.
Therefore, the injectivity of the Laplace transform implies that for any $t \geq 0$,
\begin{equation}\label{eq:gamma}
\eta_t(\lambda) = \lambda \int_t^{\infty}U_{\lambda}(dw).
\end{equation}
Here, writing $\widetilde{U}_{\lambda}(dw) = \lambda_m U_{\lambda}(dw)$, $w>0$, we have
\begin{equation}\label{eq:gamma_bar}
\eta_t(\lambda) = \int_t^{\infty}\widetilde{U}_{\lambda}(dr).
\end{equation}
Since $\eta_t(\lambda)$ is decreasing in $t$, $\eta_0(\lambda)=1$ and $\lim_{t \rightarrow \infty} \eta_t(\lambda)=0$, we deduce that $\eta_t(\lambda)$ is a tail of a probability measure, i.e.~there exists a random variable $X_{\lambda}$ such that $\eta_t(\lambda) = \Prob(X_{\lambda} > t)=\int_t^{\infty}\widetilde{U}_{\lambda}(dr)$, $t >0$. Next, assuming that $\varphi$ is strongly regularly varying at $0$, i.e.~$\varphi(q) \stackrel{0}{\sim} C q^{\alpha}$, $0<\alpha<1$ for some constant $C>0$, using \eqref{eq:gamma_laplace}, we obtain
\begin{equation*}
1 - \mathcal{L}_{\widetilde{U}_{\lambda}}(q) = 1 - \int_0^{\infty}e^{-qr}\widetilde{U}_{\lambda}(dr) = \frac{\varphi(q)}{\lambda+\varphi(q)} \stackrel{0}{\sim}  q^{\alpha}.
\end{equation*}
Then, it follows from a Tauberian theorem, see e.g.~\cite[Corollary 8.1.7]{bingham}, that equivalently we have
\begin{equation*}
\eta_t(\lambda) \stackrel{t \rightarrow \infty}{\sim} \frac{t^{-\alpha}}{\Gamma(1-\alpha)}.
\end{equation*}
Thus, for any $a>0$, we get that
\begin{equation*}
\lim_{t \rightarrow \infty} \frac{\eta_{\log(at)}(\lambda)}{\eta_{\log t}(\lambda)} = \lim_{t \rightarrow \infty} \frac{(\log a + \log t)^{-\alpha}}{(\log t)^{-\alpha}} = 1.
\end{equation*}
Therefore, $t \mapsto\eta_{\log(t)}(\lambda)$ is slowly varying at infinity, and thus $t \mapsto\eta_t(\lambda)$ is long-tailed, see e.g.~\cite[Lemma 2.15]{foss_long_tailed}, which completes the proof of the proposition. $\blacksquare$

\subsection{Proof of Theorem \ref{thm:t-ch}}\label{sec:proof_t_ch}
Writing  for $t,s \geq 0$, $H_{t,s}(u,v) = \Prob (L_t \le u, L_s \le v)$,  we have that the independence of $X$ and $L$ entails that for any  $m,n \in \NN$,
\begin{eqnarray*}
\rho_{\nu}(\Poly_m(X_{L_t}), \Poly_n(X_{L_s})) &=& \int_0^{\infty}\int_0^{\infty} \rho_{\nu}(\Poly_m(X_u), \Poly_n(X_v)) H_{t,s}(du,dv), \\
\rho_{\nu}(\Poly_m(X_{L_t}), \Nu_n(X_{L_s})) &=& \int_0^{\infty}\int_0^{\infty} \rho_{\nu}(\Poly_m(X_u), \Nu_n(X_v)) H_{t,s}(du,dv).
\end{eqnarray*}
Next, recalling that $c_{\nu}$ is symmetric, i.e.~$c_{\nu}(n,m) = c_{\nu}(m,n)$ for any $m,n \in \mathtt{N}$, Theorem \ref{thm:corr} gives that for any $u,v \geq 0$,
\begin{equation*}
\rho_{\nu}(\Poly_m(X_u), \Poly_n(X_v)) = c_{\nu}(n,m) \left( e^{-\lambda_m(u-v)} 1_{ \{u > v \} } + e^{-\lambda_n(v-u)} 1_{ \{ u \le v \} } \right)
\end{equation*}
and hence
\begin{eqnarray}
\rho_{\nu}(\Poly_m(X_{L_t}), \Poly_n(X_{L_s})) &=& \int_0^{\infty}\int_0^{\infty} c_{\nu}(n,m) \left( e^{-\lambda_m(u-v)} 1_{ \{u > v \} } + e^{-\lambda_n(v-u)} 1_{ \{u \le v \} } \right)H_{t,s}(du,dv) \nonumber \\
&=& c_{\nu}(n,m) \int_0^{\infty}\int_0^{\infty} \left( e^{-\lambda_m(u-v)} 1_{ \{u > v \} } + e^{-\lambda_n(v-u)} 1_{ \{u \le v \} }  \right)H_{t,s}(du,dv) \label{eq:inv_corr}.
\end{eqnarray}

Now, let $(u,v) \mapsto F(u,v)$ be a function of bounded variation such that $u \mapsto F(u,v)$ and $v \mapsto F(u,v)$ are also of bounded variation. Then, writing $F(du,v)$, $F(u,dv)$ and $F(du,dv)$, we mean the one dimensional measures generated by the sections $u \mapsto F(u,v)$, $v \mapsto F(u,v)$ and the two dimensional measure generated by $(u,v) \mapsto F(u,v)$ respectively. For such a function $F$, recall the bivariate integration by parts formula
\begin{eqnarray}
\int_0^{\infty}\int_0^{\infty} F(u,v)H_{t,s}(du,dv) &=& \int_0^{\infty}\int_0^{\infty} H_{t,s}([u,\infty]\times [v,\infty])F(du,dv) \nonumber \\
&+&  \int_0^{\infty}H_{t,s}([u,\infty]\times (0,\infty])F(du,0) \nonumber \\
&+& \int_0^{\infty} H_{t,s}((0,\infty]\times [v,\infty])F(0,dv) \nonumber \\
&+& F(0,0)H_{t,s}((0,\infty]\times (0,\infty]),  \label{eq:int_by_parts}
\end{eqnarray}
see e.g.~Gill et al.~\cite[Lemma 2.2]{gill1993inefficient}.
Let us apply this formula to
\[F(u,v) =  e^{-\lambda_m(u-v)} 1_{ \{u > v \} } + e^{-\lambda_n(v-u)} 1_{ \{u \le v \} }, \quad (u,v)\in \R_+^2, \]
which is clearly of bounded variation. Then, writing $\overline{H}_{t,s}(u,v) = \Prob (L_t \geq u, L_s \geq v)$  and $\overline{H}_{t}(u) = \Prob (L_t \geq u)$,
\begin{eqnarray}
\int_0^{\infty}\int_0^{\infty} F(u,v)H_{t,s}(du,dv) &=& \int_0^{\infty}\int_0^{\infty} \overline{H}_{t,s}(u,v) F(du,dv) + \int_0^{\infty} \overline{H}_{t}(u)F(du,0) \nonumber \\
&+& \int_0^{\infty} \overline{H}_{s}(v) F(0,dv) + 1, \label{eq:int_by_parts_F}
\end{eqnarray}
where we used that as $\Prob(L_t=0)=0$, $\Prob(L_t > 0)=1$ for all $t > 0$, $F(0,0)=1$ and $H$ is a distribution function.
Note that $F(du,v) = \left(-\lambda_m e^{-\lambda_m(u-v)} 1_{ \{u > v \} } + \lambda_n e^{-\lambda_n(v-u)} 1_{ \{u \le v \} } \right)du$ for all $v \geq 0$.
Thus, an integration by parts yields that
\begin{eqnarray*}
\int_0^{\infty} \overline{H}_{t}(u)F(du,0) &=& \int_0^{\infty}(1 - \Prob(L_t < u))(-\lambda_m e^{-\lambda_m u})du \\
&=& \left. e^{-\lambda_m u} \overline{H}_{t}(u)  \right\rvert_{0}^{\infty} + \int_0^{\infty} e^{-\lambda_m u} l_t(u)du
= \eta_{t}(\lambda_m) - 1,
\end{eqnarray*}
and similarly,
\begin{equation}
 \int_0^{\infty} \overline{H}_{s}(v) F(0,dv) =\eta_{s}(\lambda_n) - 1.
\end{equation}
Hence, \eqref{eq:inv_corr} reduces to
\begin{equation*}
c_{\nu}(n,m)\int_0^{\infty}\int_0^{\infty} F(u,v)H_{t,s}(du,dv) =c_{\nu}(n,m)( I(t,s) +  \eta_{t}(\lambda_m)+\eta_{s}(\lambda_n) - 1),
\end{equation*}
where we have set
\begin{equation}
I(t,s) = \int_0^{\infty}\int_0^{\infty} \overline{H}_{t,s}(u,v) F(du,dv).
\end{equation}
Then, observing that $I(s,t)=I(t,s)$, we assume, without loss of generality, that $s \le  t$ and we write $I(t,s) = I_1(t,s) + I_2(t,s) + I_3(t,s)$, where
\begin{eqnarray*}
I_1(t,s) &=& \int_0^{\infty}\int_0^v \overline{H}_{t,s}(u,v) F(du,dv), \:
I_2(t,s) = \int_0^{\infty} \int_{u = v} \overline{H}_{t,s}(u,v) F(du,dv),\\
I_3(t,s) &=& \int_0^{\infty}\int_v^{\infty} \overline{H}_{t,s}(u,v) F(du,dv).
\end{eqnarray*}
Then, as the inverse of the subordinator $\Tau$ is non-decreasing, $\overline{H}_{t,s}(u,v) = \Prob (L_t \geq u, L_s \geq v) = \Prob (L_s \geq v) = \overline{H}_{s}(v)$ for $u \le v$ and $F(du,dv) = -\lambda_n^2 \ e^{-\lambda_n(v-u)}du dv$ for $u < v$. Thus,
\begin{eqnarray*}
I_1(t,s) &=& \int_0^{\infty}\int_0^v \overline{H}_{s}(v) F(du,dv)\\
&=& -\lambda_n^2 \int_0^{\infty}\int_0^{v}\overline{H}_{s}(v)e^{\lambda_n(u-v)}dudv \\
&=& -\lambda_n \int_0^{\infty} \overline{H}_{s}(v) \left(1-e^{-\lambda_n v}\right)dv \\
&=&  -\lambda_n \int_0^{\infty} \overline{H}_{s}(v)dv +\lambda_n  \int_0^{\infty}e^{-\lambda_n v} \overline{H}_{s}(v)dv   \\
&=& -\lambda_n \E[L_s] + \lambda_n \int_0^{\infty}e^{- \lambda_n v} \Prob(L_s \geq v)dv \\
&=& -\lambda_n \E[L_s]  - \eta_{s}(\lambda_n) + 1,
\end{eqnarray*}
where in the last identity we have performed an integration by parts. We also note that
\begin{equation}\label{eq:U(0,s)}
\E[L_s]  = \int_0^{\infty} \overline{H}_{s}(v)dv = \int_0^{\infty} \Prob (L_s \geq v)dv = U(0,s).
\end{equation}
Next, writing simply $f_v(u)du = F(du,v)= \left(-\lambda_m \ e^{-\lambda_m(u-v)} 1_{ \{u > v \} } + \lambda_n\ e^{-\lambda_n(v-u)} 1_{ \{u \le v \} } \right)du$, we remark that the mapping $u \mapsto f_v(u)$ has a jump of size $(\lambda_m+\lambda_n)$ at the point $u=v$. Then,
\begin{equation*}
I_2(t,s) = \int_0^{\infty} \int_{u = v} \overline{H}_{s}(v) F(du,dv) = (\lambda_m+\lambda_n)\int_0^{\infty} \overline{H}_{s}(v)dv = (\lambda_m+\lambda_n)\E[L_s].
\end{equation*}
Finally, as $F(du,dv) = -\lambda_m^2 e^{-\lambda_m(u-v)}du dv$ for $u > v$, we deduce that
\begin{equation*}
I_3(t,s) = -\lambda_m^2 \int_0^{\infty} \overline{H}_{t,s}(u,v) \int_{v}^{\infty}e^{-\lambda_m(u-v)}dudv,
\end{equation*}
and we proceed by computing the joint tail distribution of the pair $(L_t, L_s)$, that is $\overline{H}_{t,s}(u,v) = \Prob (L_t \geq u, L_s \geq v)$.
Note that since $L$ is the inverse of $\Tau$, then $\{L_t \geq u \} = \{\Tau_u \le  t \}$, and thus $\Prob (L_t \geq u, L_s \geq v)=\Prob (\Tau_u \le t, \Tau_v \le s)$. Now, since as a L\'evy process $\Tau$ has stationary and independent increments, it follows, recalling that $s \le t$,
\begin{eqnarray*}
\overline{H}_{t,s}(u,v) = \Prob (\Tau_u \le t, \Tau_v \le s) &=& \Prob ( (\Tau_u-\Tau_v)+\Tau_v \le t, \Tau_v \le s) \\
&=& \int_0^{s} \Prob(\Tau_v \in dr) \int_0^{t-r} \Prob(\Tau_{u - v} \in dw).
\end{eqnarray*}
Using Fubini's theorem and performing the change of variable $z = u-v$, we get
\begin{eqnarray}
I_3(t,s) &=& -\lambda_m^2 \int_0^{\infty} \int_0^{s} \Prob(\Tau_v \in dr) \int_0^{t-r} \Prob(\Tau_{u - v} \in dw) \int_{v}^{\infty}e^{-\lambda_m(u-v)}dudv \nonumber \\
&=&  -\lambda_m^2 \int_0^{s} \int_0^{t-r} \int_{0}^{\infty}e^{-\lambda_m z}\Prob(\Tau_{z} \in dw)dz \int_0^{\infty} \Prob(\Tau_v \in dr) dv \nonumber \\
&=& -\lambda_m^2 \int_0^{s} \int_0^{t-r} \int_{0}^{\infty}e^{-\lambda_m z}\Prob(\Tau_{z} \in dw)dz U(dr) \label{eq:I3}
\end{eqnarray}
where in the last step, from the definition of the renewal measure, we have used that $ \int_0^{\infty} \Prob(\Tau_v \in dr)dv = U(dr)$.
Now, taking $\lambda=\lambda_m$ in ~\eqref{eq:gamma_m} and~\eqref{eq:gamma} in the proof of Proposition~\eqref{prop:long_tailed}, we have $U_{\lambda_m}(dw) = \int_{0}^{\infty}\Prob(\Tau_z \in dw)e^{-\lambda_m z}dz$, $w > 0$ and $\eta_t(\lambda_m) = \lambda_m \int_t^{\infty}U_{\lambda_m}(dw)$.
Hence, using~\eqref{eq:U(0,s)}, the expression of $I_3(t,s)$ in \eqref{eq:I3} reduces to
\begin{eqnarray*}
I_3(t,s) &=&  -\lambda_m^2 \int_0^{s} \int_0^{t-r} U_{\lambda_m}(dw) U(dr) \\
&=& \lambda_m  \int_0^{s}\eta_{t-r }(\lambda_m)U(dr) - \lambda_m \int_0^{s}U(dr)\\
&=& \lambda_m  \int_0^{s}\eta_{t-r }(\lambda_m)U(dr) - \lambda_m \E[L_s].
\end{eqnarray*}
Finally, putting all pieces together, we obtain that for $t \geq s >0$,
\begin{eqnarray}
\rho_{\nu}(\Poly_m(X_{L_t}), \Poly_n(X_{L_s})) &=& c_{\nu}(n,m) ( I_1(t,s) + I_2(t,s) + I_3(t,s) + \eta_{t}(\lambda_m) + \eta_{s}(\lambda_n) - 1 ) \nonumber \\
&=&c_{\nu}(n,m) \Bigl( -\lambda_n \E[L_s]  - \eta_{s}(\lambda_n) + 1 + (\lambda_m +\lambda_n )\E[L_s] \nonumber \\
&+& \lambda_m \int_0^{s}\eta_{t-r}(\lambda_m)U(dr) - \lambda_m \E[L_s] + \eta_{t}(\lambda_m) + \eta_{s}(\lambda_n) - 1 \Bigr) \nonumber \\
&=&c_{\nu}(n,m) \Bigl(  \lambda_m \int_0^{s}\eta_{t-r}(\lambda_m)U(dr) + \eta_{t}(\lambda_m) \Bigr) \label{eq:rho_result}
\end{eqnarray}
which provides the claim~\eqref{eq:corr_t_ch}.
We proceed by studying the spectral projections correlation structure of $\rho_{\nu}(\Poly_m(X_{L_t}), \Nu_n(X_{L_s}))$ for $t \geq s > 0$ and $m,n \in \NN$. Note that, for any $u,v \geq 0$,
\begin{equation}\label{eq:F1-2}
\rho_{\nu}(\Poly_m(X_u), \Nu_n(X_v)) =  F_1(u,v) + F_2(u,v),
\end{equation}
where we have written, see Theorem \ref{thm:corr},
\begin{eqnarray*}
F_1(u,v) &=& \rho_{\nu}(\Poly_m(X_u), \Nu_n(X_v)) 1_{ \{u \geq v \} } =\kappa_{\nu}^{-1}(m) \delta_{mn}   e^{-\lambda_m(u-v)} 1_{ \{u \geq v \} },\\
F_2(u,v) &=& \rho_{\nu}(\Poly_m(X_u), \Nu_n(X_v)) 1_{ \{ u < v \} }.
\end{eqnarray*}
Then, for any $t\geq s >0$ and $m,n \in \NN$, we have
\begin{equation}\label{eq:rho}
\rho_{\nu}(\Poly_m(X_{L_t}), \Nu_n(X_{L_s})) = \int_0^{\infty}\int_0^{\infty} F_1(u,v)H_{t,s}(du,dv) + \int_0^{\infty}\int_0^{\infty} F_2(u,v)H_{t,s}(du,dv).
\end{equation}
Now, recalling the bivariate integration by parts formula \eqref{eq:int_by_parts}, one has
\begin{eqnarray*}
\int_0^{\infty}\int_0^{\infty} F_1(u,v)H_{t,s}(du,dv) &=&  \int_0^{\infty}\int_0^{\infty} \overline{H}_{t,s}(u,v) F_1(du,dv) + \int_0^{\infty} \overline{H}_{t}(u)F_1(du,0)\\
&+& \kappa_{\nu}^{-1}(m) \delta_{mn},
\end{eqnarray*}
where we used that $ \int_0^{\infty} \overline{H}_{s}(v) F_1(0,dv) = 0$ and $F_1(0,0) = \kappa_{\nu}^{-1}(m) \delta_{mn}$. Now, following the same pattern as in the proof of the first part of Theorem~\ref{thm:corr} above, and since on $\{u<v\}$, $F_1(du,dv)=0$, one gets
\begin{eqnarray*}
\int_0^{\infty}\int_0^{\infty} F_1(u,v)H_{t,s}(du,dv) &=& \int_0^{\infty} \int_{u = v} \overline{H}_{t,s}(u,v) F_1(du,dv) +\int_0^{\infty}\int_{v}^{\infty} \overline{H}_{t,s}(u,v) F_1(du,dv) \\
&+& \int_0^{\infty} \overline{H}_{t}(u)F_1(du,0) + \kappa_{\nu}^{-1}(m) \delta_{mn} \\
&=& \kappa_{\nu}^{-1}(m) \delta_{mn} \lambda_m\E[L_s] \\ &+& \kappa_{\nu}^{-1}(m) \delta_{mn} \left( \lambda_m \int_0^s \eta_{t-r}(\lambda_m)U(dr)-\lambda_m \E[L_s]\right) \\
&+& \kappa_{\nu}^{-1}(m) \delta_{mn} \left( \eta_{t}(\lambda_m) - 1 \right)\\
&=& \kappa_{\nu}^{-1}(m) \delta_{mn} \left( \lambda_m \int_0^s \eta_{t-r}(\lambda_m)U(dr)+\eta_t (\lambda_m)\right).
\end{eqnarray*}
Next, we turn to the computation of the second integral on the right-hand side of \eqref{eq:rho}. As the functions $(u,v) \mapsto F_2(u,v)$, $u \mapsto F_2(u,v)$ and $v \mapsto F_2(u,v)$ are of bounded variation since by \eqref{eq:F1-2}, $F_2(u,v)$ is a difference of two functions of bounded variation, then, by means of the bivariate integration by parts formula \eqref{eq:int_by_parts}, we get
\begin{equation}\label{eq:F2}
\int_0^{\infty}\int_0^{\infty} F_2(u,v)H_{t,s}(du,dv) = \int_0^{\infty}\int_0^{\infty} \overline{H}_{t,s}(u,v) F_2(du,dv) + \int_0^{\infty} \overline{H}_{s}(v) F_2(0,dv),
\end{equation}
where we used that  $\int_0^{\infty} \overline{H}_{t}(u) F_2(du,0) = 0$ and  $F_2(0,0) = 0$. Now, since on $\{u > v\}$, $F_2(du,dv)=0$, then, for $t \geq s$, \eqref{eq:F2} reduces to
\begin{eqnarray*}
\int_0^{\infty}\int_0^{\infty} F_2(u,v)H_{t,s}(du,dv) &=& \int_0^{\infty} \int_0^v \overline{H}_{t,s}(u,v) F_2(du,dv) \\
&+& \int_0^{\infty} \int_{u = v} \overline{H}_{t,s}(u,v) F_2(du,dv) + \int_0^{\infty} \overline{H}_{s}(v) F_2(0,dv) \\
&=& \int_0^{\infty} \int_0^v \overline{H}_{s}(v) F_2(du,dv) + \int_0^{\infty}  \int_{u = v} \overline{H}_{s}(v) F_2(du,dv) \\
&+& \int_0^{\infty} \overline{H}_{s}(v) F_2(0,dv).
\end{eqnarray*}
Thus, $\int_0^{\infty}\int_0^{\infty} F_2(u,v)H_{t,s}(du,dv)$ does not depend on $t$. On the other hand, taking $t=s$ in \eqref{eq:rho}, we get, for any $s \geq 0$,
\begin{equation*}
\kappa_{\nu}^{-1}(m) \delta_{mn} = \kappa_{\nu}^{-1}(m) \delta_{mn} + \int_0^{\infty}\int_0^{\infty} F_2(u,v)H_{s,s}(du,dv),
\end{equation*}
wher we used that by taking $t=s$ in \eqref{eq:rho_result},  Lemma~\ref{lemma:corr} yields that for any $t \geq 0$,
\begin{equation*}
\lambda_m \int_0^{t}\eta_{t-r}(\lambda_m)U(dr) + \eta_{t}(\lambda_m) = 1.
\end{equation*}
This can also be independently proven as in Remark \ref{rem:1}.
Hence, for any $s \geq 0$,
\begin{equation*}
\int_0^{\infty}\int_0^{\infty} F_2(u,v)H_{s,s}(du,dv) = 0,
\end{equation*}
and we deduce that for any $t \geq s> 0$,
\begin{equation*}
\int_0^{\infty}\int_0^{\infty} F_2(u,v)H_{t,s}(du,dv) = \int_0^{\infty}\int_0^{\infty} F_2(u,v)H_{s,s}(du,dv) = 0.
\end{equation*}
Therefore, putting pieces together, \eqref{eq:rho} reduces to
\begin{equation*}
\rho_{\nu}(\Poly_m(X_{L_t}), \Nu_n(X_{L_s})) = \kappa_{\nu}^{-1}(m) \delta_{mn} \left( \lambda_m \int_0^s \eta_{t-r}(\lambda_m)U(dr)+\eta_t (\lambda_m)\right).
\end{equation*}
Now we are ready to study the right-hand side of \eqref{eq:corr_t_ch}  and \eqref{eq:corr_t_ch_nu} for large $t$ when $s>0$ is fixed under the assumption that $\lim_{t \rightarrow \infty}\frac{\eta_{t-s}(\lambda_m)}{\eta_{t}(\lambda_m)} = 1$. Since $t \mapsto \eta_{t}(\lambda_m)$ is decreasing on $\R^+$, we have
\begin{eqnarray*}
\int_0^{s}\eta_{t-r}(\lambda_m)U(dr) \geq   \eta_{t}(\lambda_m) U(0,s) = \eta_{t}(\lambda_m) \E[L_s]
\end{eqnarray*}
and
\begin{eqnarray*}
\int_0^{s}\eta_{t-r}(\lambda_m)U(dr) \le  \eta_{t-s}(\lambda_m)U(0,s) \le \eta_{t-s}(\lambda_m)\E[L_s].
\end{eqnarray*}
Consequently,
\begin{eqnarray*}
c_{\nu}(n,m) \eta_t(\lambda_m)(\lambda_m \E[L_s]+1) \le &\rho_{\nu}(\Poly_m(X_{L_t}), \Poly_n(X_{L_s}))& \le c_{\nu}(n,m) \eta_{t-s}(\lambda_m)(\lambda_m \E[L_s]+1) ,\\
\kappa_{\nu}^{-1}(m)\delta_{mn} \eta_t(\lambda_m)(\lambda_m \E[L_s]+1) \le &\rho_{\nu}(\Poly_m(X_{L_t}), \Nu_n(X_{L_s}))& \le \kappa_{\nu}^{-1}(m)\delta_{mn} \eta_{t-s}(\lambda_m)(\lambda_m \E[L_s]+1).
\end{eqnarray*}
Now, if for a fixed $s > 0$, there exists a constant $C=C(s,\lambda_m)>0$ such that  $\overline{\lim}_{t \rightarrow \infty}\frac{\eta_{t-s}(\lambda_m)}{\eta_t(\lambda_m)} =C$, then there exists $t_0 >0$ such that for $t \geq t_0$, $\frac{\eta_{t-s}(\lambda_m)}{\eta_t(\lambda_m)} \le C$, and thus
\begin{eqnarray*}
\rho_{\nu}(\Poly_m(X_{L_t}), \Poly_n(X_{L_s})) &\stackrel{t_0}{\asymp} & c_{\nu}(n,m) \eta_{t}(\lambda_m)(\lambda_m \E[L_s]+1) ,\\
\rho_{\nu}(\Poly_m(X_{L_t}), \Nu_n(X_{L_s}))& \stackrel{t_0}{\asymp} & \kappa_{\nu}^{-1}(m)\delta_{mn} \eta_{t}(\lambda_m)(\lambda_m \E[L_s]+1).
\end{eqnarray*}
In particular, if for a fixed $s>0$, $\lim_{t \rightarrow \infty}\frac{\eta_{t-s}(\lambda_m)}{\eta_{t}(\lambda_m)} \equiv 1$, we have
\begin{eqnarray*}
\rho_{\nu}(\Poly_m(X_{L_t}), \Poly_n(X_{L_s})) &\stackrel{t \rightarrow \infty}{\sim} &  c_{\nu}(n,m)\eta_{t}(\lambda_m)(\lambda_m \E[L_s]+1), \\
\rho_{\nu}(\Poly_m(X_{L_t}), \Nu_n(X_{L_s})) &\stackrel{t \rightarrow \infty}{\sim} & \kappa_{\nu}^{-1}(m) \delta_{mn} \eta_{t}(\lambda_m)(\lambda_m \E[L_s]+1),
\end{eqnarray*}
and this concludes the proof of Theorem~\ref{thm:t-ch}. $\blacksquare$


\begin{remark}\label{rem:1}
\normalfont
One can easily check that for any $n,m \in \NN$, when $t=s$ in \eqref{eq:corr_t_ch}, we have for any $t \geq 0$,
\begin{equation*}
\rho_{\nu}(\Poly_m(X_{L_t}), \Poly_n(X_{L_t})) = c_{\nu}(n,m),
\end{equation*}
i.e.~
\begin{equation*}
\lambda_m \int_0^{t}\eta_{t-r}(\lambda_m)U(dr) + \eta_{t}(\lambda_m) = 1.
\end{equation*}
Indeed, let us plug in $t=s$ in \eqref{eq:corr_t_ch}. Then, noting that the  convolution, we get that, for any $q > 0$,
\begin{eqnarray*}
\mathcal{L}_{\int_0^{\cdot}\eta_{\cdot-r}(\lambda_m)U(dr) }(q) &=& \mathcal{L}_{\eta_{\cdot}({\lambda_m})}(q) \mathcal{L}_{U}(q) \\
&=& \frac{\varphi(q)}{q(\lambda_m+\varphi(q))} \frac{1}{\varphi(q)}\\
&=& \frac{1}{q(\lambda_m+\varphi(q))}.
\end{eqnarray*}
Next, using \eqref{eq:double_Laplace}, one has
\begin{equation*}
\mathcal{L}_{\lambda_m \int_0^{\cdot}\eta_{\cdot-r}(\lambda_m)U(dr) + \eta_{\cdot}(\lambda_m)}(q) = \frac{\lambda_m}{q(\lambda_m+\varphi(q))} + \frac{\varphi(q)}{q(\lambda_m+\varphi(q))} = \frac{1}{q}.
\end{equation*}
Thus, by the injectivity of the Laplace transform we conclude that
\begin{equation*}
\lambda_m \int_0^{t}\eta_{t-r}(\lambda_m)U(dr) + \eta_{t}(\lambda_m) = 1.
\end{equation*}
\end{remark}

\bibliographystyle{plain}

\end{document}